\documentclass[11pt]{amsart}

\usepackage[T1]{fontenc}
\usepackage{lmodern}
\usepackage[margin=1in]{geometry}
\usepackage{amsmath,amssymb,amsthm,mathtools}
\usepackage{booktabs,array,tabularx,enumitem,longtable}
\usepackage{microtype}
\usepackage{xurl}
\usepackage{tikz-cd}
\usetikzlibrary{shapes.geometric}
\usepackage{comment}
\usepackage[hidelinks]{hyperref}
\newtheorem{theorem}{Theorem}[section]
\newtheorem{proposition}[theorem]{Proposition}
\newtheorem{lemma}[theorem]{Lemma}
\newtheorem{corollary}[theorem]{Corollary}
\theoremstyle{definition}
\newtheorem{definition}[theorem]{Definition}
\newtheorem{example}[theorem]{Example}
\theoremstyle{remark}
\newtheorem{remark}[theorem]{Remark}
\numberwithin{equation}{section}

\newcommand{\Q}{\mathbb Q}
\newcommand{\C}{\mathbb C}
\newcommand{\Z}{\mathbb Z}
\newcommand{\F}{\mathbb F}

\newcommand{\Hom}{\operatorname{Hom}}
\newcommand{\HomQ}{\operatorname{Hom}_{\Q}}

\newcommand{\lcmop}{\operatorname{lcm}}
\newcommand{\G}{\mathcal G}
\newcommand{\T}{\mathcal T}
\newcommand{\DO}{Derickx--Orli\'c}

\newcommand{\set}[1]{\left\lbrace #1 \right\rbrace}

\newcolumntype{Y}{>{\raggedright\arraybackslash}X}

\newcommand{\lmfdbec}[3]{\href{https://www.lmfdb.org/EllipticCurve/Q/#1/#2/#3}{#1.#2#3}}
\newcommand{\lmfdbeciso}[2]{\href{https://www.lmfdb.org/EllipticCurve/Q/#1/#2}{#1.#2}}

\definecolor{darkgreen}{rgb}{0,0.5,0}

\definecolor{blue}{rgb}{0,0,1}
\providecommand{\petar}[1]{{ \color{blue}{Petar: #1}}}
\definecolor{darkpurple}{rgb}{0.45,0,0.55}

\definecolor{mypink}{RGB}{219,112,147}

\title[$D$-Elliptic Modular Curves $X_0(N)$]
{$D$-elliptic modular curves $X_0(N)$}
\date{Last revised \today}

\begin{document}

\begin{abstract}
We present a method for determining whether there exists a degree $D$ rational map from $X_0(N)$ to some elliptic curve $E/\Q$. Moreover, we explain how to obtain a quadratic form that represents all possible degrees of such maps using the degree pairing method. This method previously required odd analytic rank and some additional technical conditions to be satisfied. In this paper, we extended the method to even rank curves as well and remove all technical conditions. We also show how to prove or disprove the existence of degree $D$ maps to elliptic curves by a lifting criterion on the lattices \(H_1(-,\Z)\)

As an application of this method, we determine all $D$-elliptic curves $X_0(N)$ for all $D\leq 100$. Interestingly, in some cases we found degree $D$ rational maps that we could not explain by degeneracy maps, quotient maps, and modular parameterisation.

\end{abstract}

\author{\sc Maarten Derickx}
\address{Maarten Derickx\\
University of Zagreb\\  
Bijeni\v{c}ka Cesta 30 \\
10000 Zagreb\\
Croatia}
\email{maarten@mderickx.nl}
\urladdr{http://www.maartenderickx.nl/}

\author{\sc Daeyeol Jeon}
\address{Daeyeol Jeon \\ Department of Mathematics Education \\ Kongju National University \\ Gongju, 32588 South Korea}
\email{dyjeon@kongju.ac.kr}

\author{\sc Yongjae Kwon}
\address{Yongjae Kwon \\ Department of Mathematics Education \\ Kongju National University \\ Gongju, 32588 South Korea}
\email{211049@kongju.ac.kr}

\author{\sc Petar Orli\'c}
\address{Petar Orli\'c \\
University of Zagreb\\  
Bijeni\v{c}ka Cesta 30 \\
10000 Zagreb\\
Croatia}
\email{petar.orlic@math.hr}


\maketitle


\section{Introduction}\label{sec:introduction}

The study of low-degree points on algebraic curves is closely connected 
with low-degree morphisms to $\mathbb P^1$ and elliptic curves; see for 
example \cite[Theorems 1.2, 1.3]{KadetsVogt}. This connection turns the question of ``which curves $X_0(N)$ admit a map of degree $D$ to an elliptic curve or $\mathbb P^1$'' into a sub-problem of determining all $X_0(N)$ that have infinitely many points of degree $D$.

Morphisms to $\mathbb P^1$ are 
measured by gonality, which has already been extensively studied, especially for modular curves \cite{Abramovich:Gonality, Kim2002, Poonen:Gonality}. In this paper we focus on morphisms to 
elliptic curves. We use the following terminology.

\begin{definition}[{\cite[Definition 3.1]{DerickxOrlic2024}}]
Let $C/\Q$ be a smooth, projective, geometrically integral curve, and let $d$ be a positive integer. The curve $C$ is \emph{$d$-elliptic over $\Q$} if there exist an elliptic curve $E/\Q$ and a morphism $C\to E$ of degree $d$. The cases $d=3,4,5$ are called \emph{trielliptic}, \emph{tetraelliptic}, and \emph{pentaelliptic}, respectively.
\end{definition}

We write $X_0(N)$ for the modular curve over $\Q$ attached to the congruence subgroup $\Gamma_0(N)$, and $J_0(N)$ for its Jacobian. All morphisms in this paper are defined over $\Q$ unless another field is specified.

The classification of $D$-elliptic curves \(X_0(N)\) is adjacent to several classical low-degree results.  Ogg~\cite{ogg1974hyperelliptic} classified the hyperelliptic curves \(X_0(N)\), Bars~\cite{Bars1999} classified the bielliptic ones, Hasegawa--Shimura and Jeon-Park~\cite{HasegawaShimura_trig,JeonPark05} treated trigonal and tetragonal \(X_0(N)\), and Jeon~\cite{Jeon2021} classified the curves \(X_0(N)\) with infinitely many cubic points. Gonality computations for \(X_0(N)\), $X_1(N)$, and $X_\Delta(N)$ provide a complementary low-degree perspective \cite{derickxVH,NajmanOrlic2023,Orlic2025Intermediate}. Furthermore, trielliptic and tetraelliptic curves \(X_1(N)\) were studied in \cite{Jeon2022,Jeon2023}.

The first result of this paper is the following theorem.
\begin{theorem}\label{thm:main1}There exists a deterministic algorithm that on input of a positive integer $D$ outputs the list of all values $N$ such that $X_0(N)$ is $D$-elliptic over $\Q$.
\end{theorem}

This theorem is perhaps surprising in the sense that it is unknown whether a similar result exists for maps of degree $D$ to $\mathbb P^1$, i.e., for computing the gonality. It is based on the degree pairing method introduced in \cite{DerickxOrlic2024}. The main computational input there is a degree pairing on $\operatorname{Hom}_{\Q}(J_0(N),E)$. The existence of this pairing was already shown in general in \cite{DerickxOrlic2024}. However, methods to compute it have only been described for a strong Weil curve $E/\Q$ of odd analytic rank under the assumption of some additional technical conditions. This was enough for the application to degree $4$ and $5$ points on $X_0(N)$ in \cite{DerickxOrlic2024} and \cite{DHJOdensitydegree}.

The main contribution of this paper is an algorithm to compute the degree pairing for $X_0(N)$ in full generality, without rank conditions and other technical conditions. 

We do not provide a detailed asymptotic runtime analysis of our algorithm, and instead content ourselves with showing it is highly practical by computing all $D$-elliptic curves $X_0(N)$ for $D \leq 100$. To be precise,

\begin{theorem}\label{thm:main2}
    Let $3\leq D \leq 100$. If $g(X_0(N))\geq2$, then $X_0(N)$ is $D$-elliptic over $\Q$ if and only if $N$ is listed in \begin{center}
        \textup{\url{https://github.com/nt-lib/D-elliptic/tree/main/results}.}
    \end{center}

    The file in the repository is over $100$ pages long. For expository reasons, here we present the results for $D=3,4,5$. \begin{itemize}
        \item $X_0(N)$ is trielliptic over $\Q$ if and only if
\[
N\in
\{22,30,33,34,38,42,45,52,54,57,63,72,73,81,98,108\}.
\]
        \item $X_0(N)$ is tetraelliptic over $\Q$ if and only if
\[
\begin{split}
N\in\{&28,30,33,34,39,40,42,44,45,48,51,52,55\text{--}58,60,62\text{--}66,68\text{--}70,72,\\
&74\text{--}78,80,82,84\text{--}86,88,90,91,94,96,98\text{--}100,104,105,108\text{--}112,117\text{--}121,\\
&123,124,126,128,135,136,141\text{--}145,147,155,159,160,171,176,184,188\}.
\end{split}
\]
        \item $X_0(N)$ is pentaelliptic over $\Q$ if and only if \[
N\in
\{38,46,54,67,75,81,89,100\}.
\]
    \end{itemize}
\end{theorem}

The computations for $d \leq 100$ only took approximately $2$ days of computation using a single core. Therefore, the above theorem can most certainly be extended by using more computing power if necessary.


We restrict the main statements to $g(X_0(N))\geq 2$ because genus $0$ curves admit no non-constant maps to elliptic curves, while genus $1$ curves are themselves elliptic curves and are governed by rational isogenies and endomorphisms. 

We use the same quadratic forms to find and exclude degree $D$ maps to elliptic curves. The target elliptic curve in the present problem need not be the strong Weil curve in its isogeny class, nor does it need to have odd analytic rank (a condition for computations in \cite{DerickxOrlic2024}). We therefore work with saturated old Hom-lattices and with rationally isogenous targets (their precise definitions can be found in Section \ref{subsec:candidate-exclusion}).

The finite computation proceeds in old coordinates (defined in Section \ref{subsec:old-degree-forms}). For each level $N$, each conductor $M\mid N$, and each relevant strong Weil curve $E$ of conductor $M$, one writes a possible map as a rational linear combination of the degeneracy maps from $X_0(N)$ to $X_0(M)$ composed with the modular parametrization of $E$.

The degree-pairing formula turns the degree of the attached old curve map into a quadratic form. The search then becomes a finite search for rational coordinate vectors with degree $D$, followed by the denominator congruence and, when the target is not \(E\), the homological lifting criterion of Lemma \ref{lem:matrix-lift}.

Figure~\ref{fig:candidate-removal-flow} summarizes the overall reduction. The precise procedure, including the hypotheses under which the Shimura-subgroup description can be used, is given in Section~\ref{subsec:procedure-summary} after both constructions have been developed.

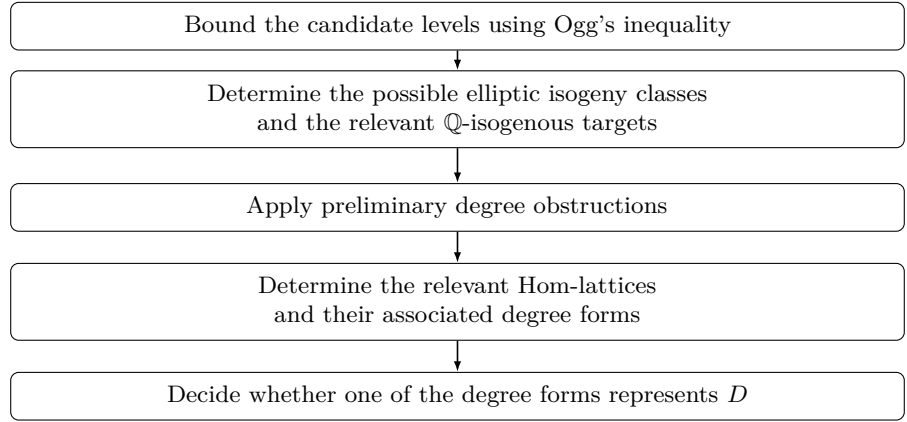
\begin{figure}[ht]
\centering
\tikzset{
overviewbox/.style={
    draw,
    rounded corners=3pt,
    align=center,
    font=\footnotesize,
    text width=.69\textwidth,
    inner xsep=6pt,
    inner ysep=5pt
},
flowarrow/.style={
    -{Latex[length=1.4mm,width=1mm]},
    semithick
}
}
\begin{tikzpicture}[x=1cm,y=1cm]
\node[overviewbox] (ogg) at (0,0)
{Bound the candidate levels using Ogg's inequality};

\node[overviewbox] (targets) at (0,-1.10)
{Determine the possible elliptic isogeny classes\\
and the relevant $\Q$-isogenous targets};

\node[overviewbox] (filters) at (0,-2.40)
{Apply preliminary degree obstructions};

\node[overviewbox] (lattices) at (0,-3.65)
{Determine the relevant Hom-lattices\\
and their associated degree forms};

\node[overviewbox] (represent) at (0,-4.90)
{Decide whether one of the degree forms represents $D$};

\draw[flowarrow] (ogg)--(targets);
\draw[flowarrow] (targets)--(filters);
\draw[flowarrow] (filters)--(lattices);
\draw[flowarrow] (lattices)--(represent);
\end{tikzpicture}

\caption{Overview of the degree $D$ existence problem.
The detailed procedure is given in Section~\ref{subsec:procedure-summary}.}
\label{fig:candidate-removal-flow}
\end{figure}

The paper is organized as follows:
\begin{itemize}
    \item In Section \ref{sec:degree-pairing}, we explain the methods for proving
    or disproving the existence of degree $D$ rational maps from $X_0(N)$ to
    an elliptic curve, via degree pairings and lattice computations. The
    subsections contain explanations and proofs of the steps in
    Figure \ref{fig:candidate-removal-flow}.
    \begin{itemize}
        \item In Section \ref{subsec:old-degree-forms}, we introduce rational
        old coordinates and the old degree form \(Q_{E,N}\). We also define
        the old-coordinate lattice \(L_{E',N}\), whose elements are precisely
        the rational old coordinates that define homomorphisms to \(E'\).

        \item In Section \ref{subsec:isogenous-lift}, we derive the degree
        equation and characterize \(L_{E',N}\) by a containment of integral
        homology images. The required containment holds precisely when the
        explicit divisibility conditions obtained from Smith normal form are
        satisfied. Column Hermite normal form then determines a \(\Z\)-basis
        of \(L_{E',N}\).

        \item In Section \ref{subsec:candidate-exclusion}, we first apply the
        modular-degree criterion and the oldform obstruction before computing
        \(L_{E',N}\). For the remaining cases, we can construct the integral
        form \(Q_{E',N}\), whose represented integers are precisely the degrees
        of morphisms to \(E'\). We also treat the case \(c_Ec_u=1\) directly. Theorem \ref{thm:pair-to-level} combines these results into a criterion
        for the existence of a morphism of degree \(D\) from \(X_0(N)\) to $E'$.

        \item In Section \ref{subsec:new_quad_forms}, under explicit hypotheses
        on the Shimura subgroup and on \(\ker\xi_{E,N}^\vee\), we compute the
        quadratic form representing all possible degrees of rational morphisms
        \(X_0(N)\to E/H\) for \(H\leq\Sigma_E\) and prove that, for a curve
        \(E'\sim_{\Q}E\), every morphism \(X_0(N)\to E'\) must factor through
        one of the curves \(E/H\). We also give a short discussion of the
        properties of \(\Sigma_E\) and \(\ker\xi_{E,N}^\vee\).

        \item In Section \ref{subsec:procedure-summary}, we assemble the
        preceding results into the procedure used in the classification and
        distinguish the preliminary obstructions from the final
        quadratic-form representation criterion.
    \end{itemize}

    \item In Section \ref{sec:positive}, we apply these methods to determine
    all $D$-elliptic curves $X_0(N)$ for $D=3,4,5$ and prove
    Theorem \ref{thm:main2} in these cases. We also provide some examples of
    quadratic forms we obtained.

    \item In Appendix \ref{app:tables}, we present
    Table \ref{tab:positive-representative-maps}, which contains explicit
    degree $D$ rational maps $X_0(N)\to E$. The maps are given by modular
    parametrizations, as well as by degeneracy and quotient maps.
\end{itemize}

\subsection{Future work}

The algorithm in Theorem \ref{thm:main1} only gives the levels $N$ for which $X_0(N)$ is $D$-elliptic. However, it does not compute explicit algebraic equations for all degree $D$ maps from $X_0(N)$ to an elliptic curve.

On the other hand, if $E$ is an elliptic curve of conductor $M\mid N$, then $\HomQ(J_0(N),E)\otimes_\Z \Q$ has an explicit basis $\set{\Phi_d}_{d \mid N/M}$ in terms of degeneracy maps; see Theorem \ref{thm:old-span}. As an intermediate step to computing all possible degrees, our algorithm produces all vectors \[\sum_{d \mid N/M} x_d \Phi_d \in \HomQ(J_0(N),E)\otimes_\Z \Q\] corresponding to degree $D$ maps to elliptic curves with respect to this basis. So, an algorithm that turns such a vector into algebraic equations for a map $f:X_0(N) \to E$ is what remains to obtain equations for all degree $D$ maps to elliptic curves.

\subsection{Data availability and reproducibility statement}
The results about $D$-elliptic curves in this paper depend on computations carried out in SageMath 10.9 \cite{SageMath}. We also used the SageMath package MD Sage (v0.1.0) which can be found at:
\begin{center}
    \url{https://github.com/koffie/mdsage/tree/v0.1.0}
\end{center}

The code for verifying all computations in this paper can be found on the GitHub repository \begin{center}
    \url{https://github.com/nt-lib/D-elliptic}.
\end{center}

The computations were performed on a server at the University of Zagreb equipped with an AMD EPYC 9175F processor (16 cores, up to 4.2 GHz) and 384 GB of RAM.

\section*{Acknowledgments}

A part of the paper was written during the authors' stay at Kongju National University and Korea Institute for Advanced Studies. We are grateful for their support and hospitality.

The first and fourth authors were supported by the Croatian Science Foundation under the project no. IP-2022-10-5008 and by the project ``Implementation of cutting-edge research and its application as part of the Scientific Center of Excellence for Quantum and Complex Systems, and Representations of Lie Algebras", PK.1.1.10.0004, European Union, European Regional Development Fund. 

The second and third authors were supported by Glocal University 30 project at Kongju National University in 2026 and by Basic Science Research Program through the National Research Foundation of Korea (NRF) funded by the Ministry of	Education (No. 2022R1A2C1010487).

\section{Quadratic forms and lattice criteria}\label{sec:degree-pairing}
We will use the following notation throughout the paper.

Let $E/\Q$ be the strong Weil curve of conductor $M\mid N$, and let
\(\pi_E:X_0(M)\to E\) be the strong Weil parametrization arising from the
newform quotient attached to $E$ by the modularity theorem
\cite[Chapter~6]{Stein2007}.

If $E'\sim_\Q E$ is an elliptic curve in the $\Q$-isogeny class of $E$, we choose an isogeny
$u:E\to E'$ that generates the rank-one $\Z$-module $\HomQ(E,E')$.
Equivalently, $u$ is a cyclic
generator of $\HomQ(E,E')$.
Let $\delta=\deg(u)$, and let $\hat u:E'\to E$ be the dual isogeny. With
minimal N\'eron differentials fixed, write
\(u^*\omega_{E'}=c_u\omega_E\).  Then \(c_u\) is a positive integer and
\(c_uc_{\hat u}=\delta\).

Write $n:=\tau(N/M)$ for the number of divisors of $N/M$, following the notation from \cite{DerickxOrlic2024}. The curve $E$ then appears with multiplicity $n$ in the $\Q$-isogeny decomposition of $J_0(N)$.

Let \(j_N:X_0(N)\to J_0(N)\) be the Abel--Jacobi map based at the cusp \(\infty\). By the universal property of the Jacobian, every morphism $f:X_0(N)\to E$ such that $f(\infty)=0$ factors uniquely through $J_0(N)$ via $j_N$. 

Let \(\Phi_E:J_0(M)\to E\) be the induced optimal quotient of Jacobians.  Thus \(E\) is the optimal elliptic quotient in its rational isogeny class.  The \emph{modular degree} of \(E\) is \(\deg(\pi_E)\).

If \(\omega_{E}\) is a N\'eron differential on \(E\), then the \emph{Manin constant} of \(\pi_E\) (see \cite{AgasheRibetStein2006}) is the positive rational number \(c_E\) satisfying \[\pi_E^*\omega_{E}=c_E\,2\pi i\,f_E(z)\,dz,\] where \(f_E\) is the normalized newform attached to \(E\). 

If $M\mid N$ and $d\mid N/M$, let \(\iota_{d,N,M}:X_0(N)\to X_0(M)\) be the standard degeneracy map, normalized as in \cite[Section~2.3]{DerickxOrlic2024}.  On the upper half-plane, \(\iota_{d,N,M}\) is induced by \(\tau\mapsto d\tau\).  On Jacobians, \(\iota_{d,N,M}\) induces a pushforward map \(\iota_{d,N,M,*}:J_0(N)\to J_0(M)\),   and hence the homomorphism \[\Phi_d=\Phi_E\circ\iota_{d,N,M,*}:J_0(N)\to E.\]  We call the maps \(\Phi_d\) the \emph{old degeneracy homomorphisms} attached to \((E,N)\).  
The word ``old'' refers to the construction of each \(\Phi_d\) from the lower-level optimal quotient \(J_0(M)\to E\) and a degeneracy map from level \(N\) to level \(M\).

Fix once and for all an order on the divisors \(d\mid N/M\), and use the resulting ordered family \((\Phi_d)_{d\mid N/M}\) as rational coordinates.

The following Theorem extends this rational-coordinate description to elliptic curves in the same \(\Q\)-isogeny class.

\begin{theorem}[{\cite{JeonKwon2026StrongWeil}}]\label{thm:old-span}
Let \(E'/\Q\) be an elliptic curve in the \(\Q\)-isogeny class of \(E\), and
let \(u:E\to E'\) be a nonzero isogeny.  Then
\[
        \HomQ(J_0(N),E')\otimes_\Z\Q
        =
        \left\langle u\circ\Phi_d : d\mid N/M\right\rangle_{\Q}.
\]
\end{theorem}
Consequently, the old-coordinate formalism for \(E\) applies to \(E'\) after replacing each \(\Phi_d\) by \(u\circ\Phi_d\).

\subsection{\texorpdfstring{Old degree form}{Old degree form}}\label{subsec:old-degree-forms}

For \(\mathbf x=(x_d)_{d\mid N/M}\in\Q^n\), define
\[\Phi_{\mathbf x}=\sum_{d\mid N/M}x_d\Phi_d \ \in\HomQ(J_0(N),E)\otimes_\Z\Q.\]
If \(\mathbf x\in\Z^n\), then
\(\Phi_{\mathbf x}:J_0(N)\to E\) is an actual homomorphism.  The problem is
to determine all \(\mathbf x\in\Q^n\setminus\Z^n\) for which
\(\Phi_{\mathbf x}\) is also an actual homomorphism.

\begin{definition}\label{def:old-degree-form}
For
\(\mathbf x\in\Z^n\), the function
\[
        \mathbf x\longmapsto
        \deg\bigl(\Phi_{\mathbf x}\circ j_N:X_0(N)\to E\bigr)
\]
is quadratic by the degree pairing construction of \cite{DerickxOrlic2024}.
Its unique extension to a quadratic form on \(\Q^n\) is the
\emph{old degree form} \(Q_{E,N}\).  Let
\(A_{E,N}\in M_n(\Z)\) be the normalized matrix of this degree
pairing, so that
\begin{equation}\label{eq:old-degree-matrix}
        Q_{E,N}(\mathbf x)
        =
        \deg(\pi_E)\,{\mathbf x}^T A_{E,N}\mathbf x.
\end{equation}
\end{definition}

The following theorem gives the entries of \(A_{E,N}\) and proves that \(Q_{E,N}\) is positive definite.

\begin{theorem}[{\cite[Proposition~2.5 and Theorem~2.13]{DerickxOrlic2024}}]\label{thm:do-degree-pairing}
The old degree form \(Q_{E,N}\) is positive definite.  Assume in addition that \(N/M\) is either squarefree or coprime to \(M\).  For divisors \(d,e\mid N/M\), put \(g=\gcd(d,e)\) and
\(\ell=\lcmop(d,e)\).  Then
\[
        (A_{E,N})_{d,e}
        =
        S_{E}(d/g)S_{E}(e/g)\frac{\psi(N)}{\psi(M\ell/g)},
\]
where \(S_{E}(n)=\sum_{m^2\mid n}\mu(m)a_{n/m^2}\) and \(\sum a_nq^n\) is the newform associated to $E$.
\end{theorem}

In general, for $E' \sim_{\mathbb{Q}} E$ and a generator $u : E \to E'$, a rational old-coordinate vector
$\mathbf x \in \mathbb{Q}^n$ need not define an actual homomorphism $u \circ \Phi_{\mathbf x} \colon J_0(N) \to E'$.
The following lattice records precisely those vectors that do.

\begin{definition}\label{def:candidate-tuple-vector}

For fixed \(E'\sim_\Q E\) and a generator \(u:E\to E'\), define the
\emph{old-coordinate lattice of \(E'\) at level \(N\)}
\[
        L_{E',N}
        :=
        \left\{\mathbf x\in\Q^n \,|\,
        u\circ\Phi_{\mathbf x}\in\HomQ(J_0(N),E')\right\}.
\]

\end{definition}

Theorem \ref{thm:old-span} implies that \(L_{E',N}\) is a free
abelian group of rank \(n\) and spans \(\Q^n\) over \(\Q\).  Thus the
remaining task is to determine \(L_{E',N}\) inside \(\Q^n\).

\subsection{Degree equation and lifting through an isogeny}\label{subsec:isogenous-lift}
Fix a rational isogeny \(u:E\to E'\) of degree \(\delta\), and let
\(\hat u:E'\to E\) be the dual isogeny.
A rational old coordinate vector \(\mathbf x\) determines
an element \(u\circ\Phi_{\mathbf x}\) of
\(\HomQ(J_0(N),E')\otimes_\Z\Q\).  The lifting problem is to determine
whether \(u\circ\Phi_{\mathbf x}\) belongs to \(\HomQ(J_0(N),E')\), equivalently whether
\(\mathbf x\in L_{E',N}\).  Every vector in \(L_{E',N}\) that gives a
degree \(D\) morphism satisfies \eqref{eq:degree-equation}.

\begin{proposition}\label{prop:degree-equation}
Let \(\mathbf x=\mathbf b/\nu\), with
\(\mathbf b=(b_d)_{d\mid N/M}\in\Z^n\) and \(\nu>0\).  Suppose that
\(G=u\circ\Phi_{\mathbf x}\) is a homomorphism \(J_0(N)\to E'\) and that
\(G\circ j_N\) has degree \(D\).  Then
\begin{equation}\label{eq:degree-equation}
        \delta Q_{E,N}(\mathbf b)=D\nu^2.
\end{equation}
\end{proposition}

\begin{proof}
Since \(\hat u\circ u=[\delta]_E\), we have
\[
        \hat u\circ G=[\delta]_E\circ\Phi_{\mathbf x}=\Phi_{\delta\mathbf x}.
\]
Thus \(\hat u\circ G\) has old coordinates \(\delta\mathbf b/\nu\).  Hence
the degree pairing gives
\[
        \deg(\hat u\circ G\circ j_N)
        =
        Q_{E,N}(\delta\mathbf b/\nu)
        =
        \frac{\delta^2Q_{E,N}(\mathbf b)}{\nu^2}.
\]
On the other hand, \(\deg(\hat u\circ G\circ j_N)=\delta D\).  Comparing the
two expressions gives \eqref{eq:degree-equation}.
\end{proof}

We call \eqref{eq:degree-equation} the \emph{degree equation}.  Equation
\eqref{eq:degree-equation} is necessary but not sufficient.  One must also decide
whether the homomorphism to \(E\) represented by
\(\delta\mathbf b/\nu\) factors through the dual isogeny
\(\hat u:E'\to E\).  Lemma \ref{lem:isogenous-lift} expresses this factorization condition in
singular homology.

\begin{lemma}\label{lem:isogenous-lift}
Let \(u:E\to E'\) be a rational isogeny, and let \(\hat u:E'\to E\) be the dual isogeny.  Composition with \(\hat u\) gives an injection \[\hat u\circ -:\HomQ(J_0(N),E')\hookrightarrow \HomQ(J_0(N),E).\]  A homomorphism \(h:J_0(N)\to E\) belongs to \(\{\hat u\circ h':h'\in\HomQ(J_0(N),E')\}\) if and only if
\begin{equation}\label{eq:isogenous-lift-containment}
	        h_*H_1(J_0(N)(\C),\Z)\subseteq
	        \hat u_*H_1(E'(\C),\Z)\subseteq H_1(E(\C),\Z).
\end{equation}
Here \(\hat u_*:H_1(E'(\C),\Z)\to H_1(E(\C),\Z)\) is the homology map induced by \(\hat u:E'\to E\).
\end{lemma}

\begin{proof}
First, composition with \(\hat u\) is injective.  Indeed, if
\(f:J_0(N)\to E'\) satisfies \(\hat u\circ f=0\), then the connected
subgroup \(f(J_0(N))\) is contained in the finite group \(\ker(\hat u)\).
Consequently, \(f(J_0(N))=\{0\}\), and hence \(f=0\).

We use the analytic representation of homomorphisms of complex abelian varieties \cite[Section~1.2]{BirkenhakeLange}.  Write
\[
        E'(\C)=V'/\Lambda',\qquad E(\C)=V/\Lambda,
\]
where \(V'\) and \(V\) are the one-dimensional complex universal-covering vector spaces and
\(\Lambda'=H_1(E'(\C),\Z)\), \(\Lambda=H_1(E(\C),\Z)\).  The isogeny \(\hat u\) is induced on universal covers by a nonzero complex-linear map \(\hat U:V'\to V\).  Since \(E'\) and \(E\) are elliptic curves, \(\hat U\) is an isomorphism of one-dimensional complex vector spaces, and \(\hat U(\Lambda')=\hat u_*\Lambda'\) is a finite-index sublattice of \(\Lambda\).

Let \(h'\) be a homomorphism from \(J_0(N)\) to \(E'\).  Choose a complex-linear lift \(H'\) of \(h'\) to universal covers.  Then \(\hat u\circ h'\) is lifted by \(\hat U\circ H'\).  For every homology class \(\lambda\in H_1(J_0(N)(\C),\Z)\), the vector \(H'(\lambda)\) lies in \(\Lambda'\).  Hence \((\hat U\circ H')(\lambda)\) lies in \(\hat u_*\Lambda'\), so \(\hat u\circ h'\) satisfies \eqref{eq:isogenous-lift-containment}.

Conversely, let \(h\) be a homomorphism from \(J_0(N)\) to \(E\) satisfying \eqref{eq:isogenous-lift-containment}.  Let \(H\) be the complex-linear lift of \(h\) to universal covers.  Define \(H'=\hat U^{-1}H\).  Condition \eqref{eq:isogenous-lift-containment} means that \(H(\lambda)\in\hat U(\Lambda')\) for every \(\lambda\in H_1(J_0(N)(\C),\Z)\).  Therefore \(H'(\lambda)\in\Lambda'\) for every such \(\lambda\).  Hence, \(H'\) descends to a homomorphism \(h'\) from \(J_0(N)\) to \(E'\).  By construction we have \(\hat u\circ h'=h\). It remains to check that \(h'\) is defined over \(\Q\).  For every
\(\sigma\in\operatorname{Gal}(\overline{\Q}/\Q)\),
\[
        \hat u\circ(\sigma h'-h')=\sigma h-h=0.
\]
The injectivity proved above gives \(\sigma h'=h'\).  Thus
\(h'\in\HomQ(J_0(N),E')\).

The forward and converse constructions prove the criterion.
\end{proof}

Lemma \ref{lem:isogenous-lift} turns the image problem for \(\hat u\circ -:\HomQ(J_0(N),E')\hookrightarrow \HomQ(J_0(N),E)\) into the lattice containment \eqref{eq:isogenous-lift-containment}.  Applied to the old coordinates, a rational coordinate vector \(\mathbf x=\mathbf b/\nu\) gives a map to \(E'\) exactly when \(\delta\mathbf x\) represents a homomorphism to \(E\) whose homology image is contained in \(\hat u_*H_1(E'(\C),\Z)\).

\begin{lemma}\label{lem:matrix-lift}
With the notation of Lemma \ref{lem:isogenous-lift}, put
\[
        \Gamma=H_1(J_0(N)(\C),\Z),\qquad
        \Lambda'=H_1(E'(\C),\Z),\qquad
        \Lambda=H_1(E(\C),\Z).
\]
Choose bases \(\gamma_1,\ldots,\gamma_r\) of \(\Gamma\), \(\lambda'_1,\lambda'_2\) of \(\Lambda'\), and \(\lambda_1,\lambda_2\) of \(\Lambda\).  Let \(U\in M_2(\Z)\) be the matrix representing \(\hat u_*:\Lambda'\to\Lambda\) with respect to the chosen bases of \(\Lambda'\) and \(\Lambda\).  Let \(h:J_0(N)\to E\) be a homomorphism over \(\Q\), and let \(B\in M_{2,r}(\Z)\) be the matrix representing \(h_*:\Gamma\to\Lambda\) with respect to the chosen bases of \(\Gamma\) and \(\Lambda\).
Then \(h=\hat u\circ h'\) for some \(h':J_0(N)\to E'\) if and only if \(UC=B\) has a solution \(C\in M_{2,r}(\Z)\).

Equivalently, the factorization \(B=U\circ C\), with \(C:\Gamma\to\Lambda'\) \(\Z\)-linear, is represented by the following commutative diagram.
\[
\begin{tikzcd}[column sep=large,row sep=large]
H_1(J_0(N)(\C),\Z) \arrow[rr, dashed, "C"] \arrow[dr, "B"'] &&
H_1(E'(\C),\Z) \arrow[dl, "U=\hat u_*"] \\
& H_1(E(\C),\Z)
\end{tikzcd}
\]
If \(P U Q=\operatorname{diag}(s_1,s_2)\) is a Smith normal form of \(U\), then \(UC=B\) has a solution in \(M_{2,r}(\Z)\) if and only if
\begin{equation}\label{eq:smith-divisibility}
        s_i\mid(PB)_{ij}\qquad (i=1,2,\ j=1,\ldots,r).
\end{equation}
\end{lemma}

\begin{proof}
Suppose first that \(h=\hat u\circ h'\) for some homomorphism \(h'\colon J_0(N)\to E'\) over \(\Q\).  Let \(C=(c_{kj})\) be the matrix of \(h'_*\), so \(h'_*(\gamma_j)=c_{1j}\lambda'_1+c_{2j}\lambda'_2\), and let \(C_j\) denote the \(j\)th column of \(C\).  Applying \(\hat u_*\), the coordinate vector of \((\hat u_*\circ h'_*)(\gamma_j)\) in \(\Lambda\) is \(UC_j\).  Since \(\hat u\circ h'=h\), the vector \(UC_j\) is also the \(j\)th column $B_j$ of \(B\).  Thus \(UC_j=B_j\) for every \(j\), which is exactly \(UC=B\).

Conversely, suppose that \(UC=B\) has a solution \(C\in M_{2,r}(\Z)\).  For each basis element \(\gamma_j\), the equality \(UC_j=B_j\) means that \(h_*(\gamma_j)=\hat u_*(c_{1j}\lambda'_1+c_{2j}\lambda'_2)\).  Therefore \(h_*(\Gamma)\subseteq \hat u_*\Lambda'\). By Lemma \ref{lem:isogenous-lift}, there exists
\(h'\in\HomQ(J_0(N),E')\) satisfying \(\hat u\circ h'=h\).  This proves the
equivalence between lifting and integral solvability of \(UC=B\).

For the Smith normal form criterion,
put \(\Delta_U=\operatorname{diag}(s_1,s_2)\), so \(P U Q=\Delta_U\)
with \(P,Q\in\operatorname{GL}_2(\Z)\). The image \(\hat u_*\Lambda'\) has finite index in \(\Lambda\), so \(U\) has rank \(2\).  We choose the Smith invariant factors \(s_1,s_2\) to be positive. Thus the equation \(UC=B\) is equivalent to
\[
\begin{gathered}
        \Delta_U(Q^{-1}C)=PB.
\end{gathered}
\]
Set \(C'=Q^{-1}C\). Since \(Q\in\operatorname{GL}_2(\Z)\), \(C\) has integer entries if and only if \(C'\) has integer entries. Thus \(UC=B\) has a solution \(C\in M_{2,r}(\Z)\) if and only if
\(\Delta_UC'=PB\) has a solution \(C'\in M_{2,r}(\Z)\). Due to the fact that \(\Delta_U\) is diagonal, its unique rational solution \(C'=\Delta_U^{-1}PB\) is integral if and only if \eqref{eq:smith-divisibility} holds.
\end{proof}

Lemma \ref{lem:matrix-lift} characterizes the elements of \(L_{E',N}\)
by the Smith normal form criterion \eqref{eq:smith-divisibility}. To obtain a
\(\Z\)-basis of \(L_{E',N}\), we use the column Hermite normal form. Cohen
\cite[Section~2.4]{cohen1} treats both normal forms for
\(\Z\)-modules.  For a matrix \(G\in M_{n,m}(\Z)\) of rank \(n\), its column Hermite
normal form is the unique lower-triangular matrix
\(H=(h_{ij})\in M_n(\Z)\) satisfying
\begin{equation}\label{eq:column-hermite-normal-form}
G\Z^m=H\Z^n,\qquad
h_{ii}>0,\qquad 0\le h_{ij}<h_{ii}\quad (j<i).
\end{equation}

\begin{lemma}\label{lem:hermite-lattice-basis}
Let \(L\subset\Q^n\) be a rank-\(n\) lattice, choose \(q>0\) such that
\(qL\subset\Z^n\), and let \(G\in M_{n,m}(\Z)\) satisfy
\(G\Z^m=qL\).  If \(\mathbf h_1,\ldots,\mathbf h_n\) are the columns of the
column Hermite normal form \(H\) of \(G\), then
\(\mathbf h_1/q,\ldots,\mathbf h_n/q\) form a \(\Z\)-basis of \(L\).
\end{lemma}

\begin{proof}
By \eqref{eq:column-hermite-normal-form}, the columns of \(H\) generate
\(qL\).  Since \(H\) is nonsingular, the columns are linearly independent and
hence form a \(\Z\)-basis of \(qL\).  Dividing by \(q\) gives the stated basis
of \(L\).
\end{proof}

Thus Smith normal form supplies the lifting conditions, whereas
Lemma \ref{lem:hermite-lattice-basis} converts the solution lattice of
\eqref{eq:smith-divisibility} into the
basis used in Proposition \ref{prop:candidate-criteria}.  Hermite reduction
introduces no additional lifting condition.

\subsection{Old-coordinate lattices}\label{subsec:candidate-exclusion}

To construct \(Q_{E',N}\), we first restrict the possible strong Weil
pairs \((M,E)\) using the modular degree of \(\pi_E\). For each remaining
\(E'\sim_\Q E\), the old form \(Q_{E,N}\) gives a further necessary
representation condition before \(L_{E',N}\) is computed. A \(\Z\)-basis
of \(L_{E',N}\) then converts the degree pairing into the integral form
\(Q_{E',N}\).

\begin{theorem}[{\cite{JeonKwon2026StrongWeil}}]\label{thm:divisibility}
Let \(E/\Q\) be the strong Weil curve of conductor \(M\), and let
\(\pi_E:X_0(M)\to E\) be the strong Weil parametrization. Suppose that \(M\mid N\) and that \(g:X_0(N)\to E'\) is a nonconstant rational map to an elliptic curve
isogenous to \(E\).  Then
\begin{equation}\label{eq:modular-degree-coarse}
        \deg(\pi_E)\mid c_E^{\Omega(N/M)}\deg(g).
\end{equation}
Here $\Omega(N/M)$ is the number of prime factors of $N/M$ counted with multiplicity.
More sharply, put
\[
\mathcal D_c(N/M)=
\begin{cases}
1,&N=M,\\[2mm]
\displaystyle\prod_{\ell\mid c}
\ell^{v_\ell(c)+\lfloor v_\ell(N/M)/2\rfloor},&N>M.
\end{cases}
\]
Then \(\deg(\pi_E)\mid\mathcal D_{c_E}(N/M)\deg(g)\).
\end{theorem}

We call the sharper necessary condition the \emph{modular-degree
criterion}.  For a degree \(D\) search, it retains only the pairs satisfying
\(\deg(\pi_E)\mid\mathcal D_{c_E}(N/M)D\).  It does not require a general
denominator bound.

\begin{remark}\label{manin_constant_remark}
When \(c_E=1\), the modular-degree criterion becomes
\(\deg(\pi_E)\mid D\).  Cremona's computations of optimality and Manin
constants \cite{Cremona:ECData,AgasheRibetStein2006} show that \(c_E=1\) for
every \(\Gamma_0(M)\)-optimal elliptic curve \(E/\Q\) of conductor \(M<400000\)
in Cremona's tables.  

A direct check of Ogg's inequality shows that, for
\(D=4166\), the largest level not excluded by the inequality is \(N=399913\),
whereas for \(D=4167\) the prime level \(N=400009\) is not excluded. Thus
\(D=4167\) is the first value for which Ogg's inequality allows a level
\(N>400000\).  Consequently, for \(D\le4166\), every elliptic curve appearing in
our computation has conductor \(M<400000\), and hence satisfies \(c_E=1\).
\end{remark}

The modular-degree criterion depends only on the pair \((M,E)\) and
does not use the chosen curve \(E'\).  Once \(E'\) and \(u\) are fixed, the
old-coordinate lattice for \(E\) yields a second exclusion test before
\(L_{E',N}\) is constructed.  The test is an integral representation problem
for \(Q_{E,N}\).

\begin{corollary}\label{cor:old-form-obstruction}
Put \(q=\mathcal D_{c_E}(N/M)\).  If a degree \(D\) morphism
\(X_0(N)\to E'\) over \(\Q\) exists, then
\begin{equation}\label{eq:old-form-obstruction}
        Q_{E,N}(\mathbf z)=\delta q^2D
        \qquad\text{for some }\mathbf z\in\Z^n.
\end{equation}
More precisely, if the associated homomorphism has reduced old
coordinate vector \(\mathbf x=\mathbf b/\nu\in L_{E',N}\), then
\(\nu\mid\delta q\) and one may take
\(\mathbf z=\delta q\mathbf x\).  In particular, when \(c_E=1\), the equation
\eqref{eq:old-form-obstruction} becomes
\(Q_{E,N}(\mathbf z)=\delta D\).
\end{corollary}

\begin{proof}
After translating the morphism so that the cusp \(\infty\) maps to
zero, let \(G=u\circ\Phi_{\mathbf x}:J_0(N)\to E'\) be the associated
homomorphism.  Since
\(\hat u\circ G=\Phi_{\delta\mathbf x}\), we apply to \(E\) the lattice
inclusion proved in \cite{JeonKwon2026StrongWeil}:
\begin{equation}\label{eq:source-old-lattice-inclusion}
q\,\HomQ(J_0(N),E)
\subseteq
\bigoplus_{d\mid N/M}\Z\Phi_d .
\end{equation}
Indeed, \(q\Phi_{\delta\mathbf x}=\Phi_{\delta q\mathbf x}\), so
\eqref{eq:source-old-lattice-inclusion} and the uniqueness of rational old
coordinates give \(\delta q\mathbf x\in\Z^n\).  If
\(\mathbf x=\mathbf b/\nu\) is reduced, this integrality implies
\(\nu\mid\delta q\).  Put \(\mathbf z=\delta q\mathbf x\).  Proposition
\ref{prop:degree-equation} gives \(D=\delta Q_{E,N}(\mathbf x)\), and hence
\[
        Q_{E,N}(\mathbf z)
        =\delta^2q^2Q_{E,N}(\mathbf x)
        =\delta q^2D.
\]

\end{proof}

If \(Q_{E,N}\) does not represent \(\delta q^2D\),
Corollary \ref{cor:old-form-obstruction} excludes \(E'\) without constructing
\(L_{E',N}\).  If \(Q_{E,N}\) does represent \(\delta q^2D\), a representing
vector \(\mathbf z\) need not satisfy
\(\mathbf z/(\delta q)\in L_{E',N}\).  Lemmas
\ref{lem:isogenous-lift} and \ref{lem:matrix-lift} therefore determine
\(L_{E',N}\), and a \(\Z\)-basis of this lattice converts \(Q_{E,N}\) into
the integral quadratic form of Proposition
\ref{prop:candidate-criteria}.

\begin{proposition}\label{prop:candidate-criteria}
Let \(\mathbf v_1,\ldots,\mathbf v_n\) be a \(\Z\)-basis of
\(L_{E',N}\).  With respect to \((\mathbf v_i)_{i=1}^n\), define
\begin{equation}\label{eq:target-degree-form}
Q_{E',N}(z_1,\ldots,z_n)
=\deg(u)\,Q_{E,N}\!\left(\sum_{i=1}^n z_i\mathbf v_i\right)
\end{equation}
Then \(Q_{E',N}\) is an integral positive-definite quadratic form.  For every positive integer
\(D\), a degree \(D\) morphism \(X_0(N)\to E'\) over \(\Q\) exists if and
only if \(Q_{E',N}\) represents \(D\) over \(\Z\).
\end{proposition}

\begin{proof}

Write \(\mathbf z=(z_i)_{i=1}^n\in\Z^n\) and put
\(\mathbf x=\sum_{i=1}^n z_i\mathbf v_i\).  By the definition of the chosen
\(\Z\)-basis, \(\mathbf z\) corresponds to the homomorphism
\(G=u\circ\Phi_{\mathbf x}:J_0(N)\to E'\), and the degree pairing gives
\begin{equation}\label{eq:target-degree-identity}
\deg(G\circ j_N)=\deg(u)Q_{E,N}(\mathbf x)=Q_{E',N}(\mathbf z).
\end{equation}
For every \(\mathbf z\in\Z^n\), the morphism \(G\circ j_N\) has
degree in \(\Z_{\geq 0}\).  Equation \eqref{eq:target-degree-identity}
therefore gives \(Q_{E',N}(\mathbf z)\in\Z\) for every
\(\mathbf z\in\Z^n\), so \eqref{eq:target-degree-form} defines an integral
quadratic form.
The form \(Q_{E',N}\) is positive
definite because \(Q_{E,N}\) is positive definite and the basis matrix is
invertible over \(\Q\).

Suppose first that \(Q_{E',N}(\mathbf z)=D\) for some
\(\mathbf z\in\Z^n\).  The homomorphism
\[G=u\circ\Phi_{\sum_i z_i\mathbf v_i}:J_0(N)\to E'\] then gives a degree \(D\) morphism
\(G\circ j_N:X_0(N)\to E'\) by \eqref{eq:target-degree-identity}.

Conversely, let \(f:X_0(N)\to E'\) be a degree \(D\) morphism over
\(\Q\).  Since the cusp \(\infty\) is rational, translation by
\(-f(\infty)\) is defined over \(\Q\).  Translation by \(-f(\infty)\)
preserves degree and produces a
morphism \(f_0\) satisfying \(f_0(\infty)=0\).  By the universal property of
the Jacobian, there is a unique homomorphism \(G:J_0(N)\to E'\) such that
\(f_0=G\circ j_N\).  Writing \(G\) in the basis \((\mathbf v_i)_{i=1}^n\) gives a vector
\(\mathbf z\in\Z^n\), and \eqref{eq:target-degree-identity} yields
\(Q_{E',N}(\mathbf z)=D\).
\end{proof}

\begin{corollary}\label{cor:fixed-target-old-basis}
If \(c_Ec_u=1\), then \(L_{E',N}=\Z^n\).  Hence a degree \(D\) morphism to
\(E'\) exists if and only if the integral form \(\deg(u)Q_{E,N}\) represents
\(D\) over $\Z$.
\end{corollary}

\begin{proof}
When \(c_Ec_u=1\), it is proved in
\cite{JeonKwon2026StrongWeil} that the homomorphisms
\(u\circ\Phi_d\), indexed by \(d\mid N/M\), form a \(\Z\)-basis of
\(\HomQ(J_0(N),E')\).  Hence their old-coordinate lattice is \(\Z^n\),
and the assertion follows from Proposition \ref{prop:candidate-criteria}.
\end{proof}

When Corollary \ref{cor:fixed-target-old-basis} does not apply,
Lemma \ref{lem:isogenous-lift} characterizes \(L_{E',N}\) by a containment of
integral homology images, and Lemma \ref{lem:matrix-lift} converts the
containment in \eqref{eq:isogenous-lift-containment} into integral linear
congruences in \(\mathbf x\).  Smith normal form tests the congruences in
Lemma \ref{lem:matrix-lift}, and Lemma
\ref{lem:hermite-lattice-basis} computes a \(\Z\)-basis of \(L_{E',N}\) from
an integer generating matrix.  Proposition \ref{prop:candidate-criteria}
then decides whether \(Q_{E',N}\) represents \(D\) over $\Z$.

The construction of \(L_{E',N}\) and \(Q_{E',N}\) yields
the following existence criterion at level \(N\).

\begin{theorem}
\label{thm:pair-to-level}
Let \(D\) be a positive integer, and fix a level \(N\).  Then \(X_0(N)\)
admits a degree \(D\)
morphism to an elliptic curve over \(\Q\) if and only if there exist a divisor
\(M\mid N\), a strong Weil curve \(E\) of conductor \(M\), and an elliptic
curve \(E'\sim_{\Q}E\), taken up to \(\Q\)-isomorphism, such that
\(Q_{E',N}\) represents \(D\) over $\Z$.
\end{theorem}
\begin{proof}
Let \(f:X_0(N)\to E'\) have degree \(D\).  After replacing
\(f\) by \(f-f(\infty)\), the universal property of the Jacobian and Theorem
\ref{thm:divisibility} determine a divisor \(M\mid N\) and a retained strong
Weil curve \(E\) with \(E'\sim_\Q E\).  Proposition
\ref{prop:candidate-criteria} shows that \(Q_{E',N}\) represents \(D\) over $\Z$.
Conversely, every such representation gives a degree \(D\) map by
Proposition \ref{prop:candidate-criteria}.
\end{proof}

\subsection{Shimura subgroups and explicit quotient lattices}\label{subsec:new_quad_forms}

Theorem~\ref{thm:pair-to-level} reduces the existence problem to representation by the forms $Q_{E',N}$. These forms depend on the integral lattices $L_{E',N}$. We now study conditions under which the Shimura subgroup gives an explicit description of these lattices and of the factorization of morphisms to isogenous elliptic curves.

The result is not a separate existence criterion, but a structural description that can simplify the construction of the same degree forms. When its hypotheses fail, the integral-homology criterion of Section~\ref{subsec:isogenous-lift} remains available. 

\begin{definition}
    Let $M$ be a positive integer and let $\pi:X_1(M)\to X_0(M)$ be the natural map $(E,P)\mapsto(E,\left<P\right>)$. The Shimura subgroup $\Sigma(M)$ is the kernel of the map $\pi^*:J_0(M)\to J_1(M)$. Furthermore, for an abelian subvariety $A \subseteq J_0(M)$ we define the Shimura subgroup of $A$ to be $A \cap \Sigma(M)$.
\end{definition}

    We will study Shimura subgroups of $E$. For simplicity, we will often write $\Sigma_E=E\cap\Sigma(M)$. Let us also define a map \[\xi_{E,N}=(\Phi_1,\ldots,\Phi_{N/M}): J_0(N)\to E^n.\] The triviality of $\ker \xi_{E,N}^\vee$ was used in \cite{DerickxOrlic2024} to determine a quadratic form for elliptic curves $E$ of odd analytic rank. In general, this kernel is not necessarily trivial. However, we can link its structure with $\Sigma_E$ and obtain similar results.

\begin{theorem}[{\cite[Theorem 4]{LING199539}}]\label{thmling}
    Let $N$ be a positive integer, and let $M$ be a divisor of $N$ such that $\frac{N}{M}$ is square--free and $\left(M,\frac{N}{M}\right)=1$. We define
    $$\Sigma(M)_0^{n}:=\left\{(x_1,\ldots,x_{n}) \,\middle|\, x_i\in \Sigma(M),\, \sum_1^{n} x_i=0\right\}.$$
    We define the map $\tau^*_{N,M}:=(\iota^*_{1,N,M},\ldots,\iota^*_{N/M,N,M}): J_0(M)^{n}\to J_0(N)$.
    \begin{enumerate}[label=(\roman*)]
        \item If $M$ is odd or $\frac{N}{M}$ is a prime, then $\ker \tau_{N,M}^* = \Sigma(M)_0^{n}$.
        \item If $M$ is even and $\frac{N}{M}$ is not a prime, then $\ker \tau_{N,M}^*$ and $\Sigma(M)_0^{n}$ are equal up to a $2$-group.
    \end{enumerate}
\end{theorem}

We know that $\Sigma(M)_0^n\cap E^n\cong \Sigma_E^{n-1}$. From looking at the proof of this theorem one can see that actually \[\Sigma(M)_0^n\cap E^n\subseteq \ker \xi_{E,N}^\vee\] for all $E$ and $N$. Hence it is enough to check the cardinalities of these groups, and this is easy to do with Sage because there is an inbuilt function \texttt{E.shimura\_subgroup()}.

We prove a structural lemma that will be used in the main results of this section. 

\begin{lemma}\label{lem:hom_to_J0_ambient}
	Let $E$ be a strong Weil curve of level $M\mid N$ and $E'\sim_\Q E$.
    Assume that $\ker\xi_{E,N}^\vee=\Sigma(M)_0^n\cap E^n$. Then there is a rational isogeny $A:=E^n/\ker \xi_{E,N}^\vee \to (E/\Sigma_E)^n$ of degree $\#\Sigma_E$. Moreover, the composition with this isogeny gives an injection $$\Hom_\Q(E', J_0(N))\cong\Hom_\Q(E',A) \hookrightarrow \Hom_\Q(E',(E/\Sigma_E)^n).$$
\end{lemma}
\begin{proof}
	Since $\ker\xi_{E,N}^\vee=\Sigma(M)_0^n\cap E^n$, which is an index $\#\Sigma_E$ subgroup of $\Sigma_E^n$, we obtain a rational isogeny $\Phi : A\to (E/\Sigma_E)^n$ of that degree. Hence, we also get a mapping \[\Hom_\Q(E',A) \rightarrow \Hom_\Q(E',(E/\Sigma_E)^n), \ f\mapsto \Phi\circ f.\] This mapping is injective because $\Hom_\Q(E',A)$ is torsion-free \cite[Lemma 12.2]{Milne1986AV}. 
    
    Finally, $A$ is the maximal $E$ isogenous subvariety of $J_0(N)$, meaning that any rational homomorphism $E\to J_0(N)$ factors uniquely through $A$ (see \cite[Section 3]{DerickxOrlic2024}). This gives us the isomorphism $\Hom_\Q(E', J_0(N))\cong\Hom_\Q(E',A)$.
\end{proof}

For the next proposition, we introduce the following notation. Let $E_1$ and $E_2$ be elliptic curves, and $f: E_1 \to E_2$ be an isogeny. Then for any $\mathbf c=(c_i) \in \Z^n$ we define \begin{equation} \label{eq:diagonal_morhpism}f_{\mathbf c} : E_1 \to E_2^n, \ x\mapsto (c_if(x))_i.
\end{equation}

\begin{proposition}\label{prop:quad_form_factorisation}
Let $E$ be a strong Weil curve of level $M\mid N$ and $E'\sim_\Q E$. Suppose that $\Sigma_E$ is cyclic and $\Sigma(M)_0^n\cap E^n=\ker \xi_{E,N}^\vee$. Then the following statements hold.
\begin{enumerate}
	\item Any rational homomorphism $E' \to J_0(N)$ factors through some elliptic curve $E''=E/H$ with $H\leq \Sigma_E$.
	\item Let $H\leq\Sigma_E$, and put $E''=E/H$ and $h=\#H$. If $f''$ denotes the isogeny $E'' \to E/\Sigma_E$, then
	$$ \Hom_\Q(E'', J_0(N)) = \set{f''_{\mathbf c} \, \middle |\, \mathbf c \in \Z^n,\, \sum_{i=1}^n c_i \equiv 0 \pmod{h}} \subseteq \Hom_\Q(E'',(E/\Sigma_E)^n).$$
\end{enumerate}
\end{proposition}
\begin{proof}
	We first introduce some notation. Define $E_s:=E/\Sigma_E$ and let $f$ denote an induced isogeny $E \to E_s$. In addition, we let $f': E'\to E_s$ be a $\Q$-rational isogeny of minimal degree. We write $E(\C)=\C/\Lambda$, $E'(\C)=\C/\Lambda'$, $E_s=\C/\Lambda_s$, and assume that the lattices have been chosen such that $\Lambda, \Lambda' \subseteq \Lambda_s$ and that $f_\C$ and $f'_\C$ are induced by the inclusions $\Lambda, \Lambda' \subseteq \Lambda_s$. Additionally, we take generators $w_1,w_2$ of $\Lambda_s$ such that $\#\Sigma_Ew_1,w_2$ are generators of $\Lambda$. Such generators exist because $\Sigma_E$ is cyclic. We let $$t : \Lambda_s^n \to \Lambda_s/\Lambda \cong \Sigma_E$$ be the sum map. Finally, we let $\Lambda_s(N)$ be the lattice such that $\textrm{Im}\, \xi_{E,N}^\vee = \C^n/\Lambda_s(N)$.
	
	The assumption that the kernel of $\xi_{E,N}^\vee$ is as expected is equivalent to 
	$$\Lambda_s(N) = \ker t.$$
	Since $f'$ was chosen of minimal degree, we know that it is cyclic and a generator of $\Hom_\Q(E',E_s)\cong\Z$. 
    In particular, any element of $\Hom_\Q(E',E_s^n)$ is of the form $f'_{\mathbf c}$ for some $\mathbf c\in \Z^n$.
	
	We start with the proof of part (1).  Let $\mathbf c=(c_1,\ldots,c_n)\in \Z^n$ and define $$\mathbf c\Lambda' := \set{(c_i\lambda') | \lambda' \in \Lambda'}\subseteq \Lambda'.$$ The condition that $f_{\mathbf c}$ is an element of $\Hom(E', J_0(N))$ is equivalent to $$\mathbf c\Lambda' \subseteq \Lambda_s(N).$$ Define $\Lambda'' = \Lambda+\Lambda'$, $E''(\C) := \C/\lambda$ and let $f'': E'' \to E_s$ be the morphism corresponding to the inclusion $\Lambda+\Lambda'$. Note that since $f$ and $f'$ are $\Q$-rational, $f''$ can also be defined over $\Q$. Furthermore, $f'$ factors through $f''$, hence $f'_{\mathbf c}$ factors through $f''_{\mathbf c}$.
    
    It remains to prove that \[f'_{\mathbf c}\in\Hom_\Q(E', J_0(N))\implies f''_{\mathbf c}\in\Hom_\Q(E'', J_0(N)).\]
	If $f'_{\mathbf c} \in \Hom_\Q(E', J_0(N))$, then $\mathbf c \Lambda' \subseteq \Lambda_s(N).$ Also $\mathbf c\Lambda \subseteq \Lambda^n \subseteq \Lambda_s(N)$, so that $\mathbf c\Lambda'' =\mathbf c\Lambda + \mathbf c\Lambda' \subseteq \Lambda_s(N)$, which implies that indeed $f''_{\mathbf c} \in \Hom_\Q(E'', J_0(N))$.
	
	Now we prove part (2). The condition $f''_{\mathbf c} \in \Hom_\Q(E'', J_0(N))$ is equivalent to $\mathbf c \Lambda'' \subseteq \Lambda_s(N)$. As $\Lambda_s(N)=\ker t$, this is equivalent to $t(\mathbf c \Lambda'')=0$. Now if $\lambda''\in \Lambda''$, then \[t(\mathbf c\lambda'') \equiv \left(\sum c_i \right )\lambda''  \pmod{\Lambda}.\] If we take $\lambda''$ such that $\lambda '' \pmod{\Lambda}$ is a generator of $\Lambda''/\Lambda \cong H$, then $\lambda''$ has order $h$. Thus it follows that $$f''_{\mathbf c} \in \Hom(E'',J_0(N)) \Leftrightarrow \left(\sum c_i \right )\lambda'' \equiv 0 \pmod{\Lambda} \Leftrightarrow \sum c_i \equiv 0 \pmod{h},
	$$
	as desired.
\end{proof}

\begin{corollary}\label{cor:new_basis}
Let $E/\Q$ be a strong Weil curve of conductor $M\mid N$. 
Suppose that $\Sigma_E$ is cyclic and $\Sigma(M)^n_0\cap E^n=\ker\xi^\vee_{E,N}$.
Let $H\leq\Sigma_E$, and put $E''=E/H$ and $h=\#H$.
Let $u:E\to E''$ be the quotient isogeny.
Then the old-coordinate lattice with respect to $u$ is
\[
L_{E'',N}=\left\{
\frac{\mathbf{c}}{h} \,\middle|\, \mathbf{c}=(c_d)_{d\mid N/M}\in\Z^n,\quad
\sum_{d\mid N/M}c_d\equiv 0\pmod h
\right\}.
\]
Consequently,
\[
\{u\circ\Phi_1\}
\cup
\left\{
\frac{u\circ\Phi_d-u\circ\Phi_1}{h}
\;\middle|\;
d\mid N/M,\ d\neq1
\right\}
\]
is a $\Z$-basis of $\HomQ(J_0(N),E'')$.
\end{corollary}

The coefficients for the quadratic form with this new basis can easily be calculated with Theorem \ref{thm:do-degree-pairing} as we are just performing a base change. We will see an example of that in Proposition \ref{prop:quad_forms} and Example \ref{ex:84}.

We also see that the \emph{old quadratic form} $Q_{E,N}$ represents all possible degrees of morphisms $X_0(N)\to E$. Indeed, in this case we have $\#H=1$ and $u=\operatorname{Id}_E$, and the basis $\{\Phi_d\}_{d\mid N/M}$ may be replaced by the basis consisting of $\Phi_1$ and $\Phi_d-\Phi_1$ for all divisors $d\mid N/M$ with $d\neq1$. These two bases are related by a unimodular change of basis, so $Q_{E,N}$ already gives the full degree spectrum of morphisms $X_0(N)\to E$.

The above results are useful because, if the conditions on $\Sigma_E$ are satisfied, then we can easily calculate quadratic forms for elliptic curves $E/H$ and any rational morphism from $X_0(N)$ to $E'$ will factor as \[X_0(N)\to E/H\to E'.\] Furthermore, the $\Q$-isogeny graphs of elliptic curves over $\Q$ are widely known and available on LMFDB \cite{LMFDB}.

\begin{remark}\label{rem:odd_rank}
    Notice that these results are an extension of results in \cite{DerickxOrlic2024}. Only odd analytic rank elliptic curves $E$ were considered there and it was computed that \[\Sigma_E=\{0\}=\ker \xi_{E,N}^\vee\] for all $\textup{Cond}(E)\mid N<408$. In that case the \emph{old quadratic form} gives all possible degrees of morphisms $X_0(N)\to E$ and any morphism $X_0(N)\to E'$ factors through $E$, as was proved in \cite[Theorem 1.12]{DerickxOrlic2024}.

    Moreover, if $\textup{Cond}(E)=N$, then the basis of $\HomQ(J_0(N),E)$ will be just $\{\Phi_1\}=\{\pi_E\},$ where $\pi_E$ is the modular parametrization of $E$, and all maps $J_0(N)\to E'$ will factor through some map $u\circ\pi_E:J_0(N)\to E/H$, and therefore through $E$. As we can see, this case directly corresponds to the Modularity theorem.
\end{remark}

\begin{remark}\label{rem:shimura_ker_not_equal}
    We may wonder how often this method can be used. Theorem \ref{thmling} tells us that these groups may differ by a $2$-group and the question is when does the condition $\Sigma(M)_0^n\cap E^n=\ker \xi_{E,N}^\vee$ hold and when is $\Sigma_E$ cyclic.

    The Shimura subgroup $\Sigma_E$ is cyclic for all elliptic curves of conductor $\textup{Cond}(E)\leq800$. Moreover, we have that $\#\Sigma_E\leq16$ for all elliptic curves $E/\Q$, see \cite[Remark 4.8]{DerickxOrlic2026Intermediate}. For odd rank elliptic curves we also know that $\Sigma_E\leq E[2]$ by \cite[Proposition 4.7]{DerickxOrlic2026Intermediate} and the Shimura subgroup was computed to be trivial if $\textup{Cond}(E)\leq800$. So, it makes sense to ask what is the possible size of $\Sigma_E$ and is it bounded by $3$ when $N\geq18$.
    
    Out of all pairs $(E,N)$ considered for determining $D$-elliptic curves $X_0(N)$ in Section \ref{sec:positive}, we had $\Sigma(M)_0^n\cap E^n\neq\ker \xi_{E,N}^\vee$ for only $4$ pairs. These are


\begin{center}
\small
\setlength{\tabcolsep}{4pt}
\begin{tabular}{cccccc}
\toprule
$E$ & $N$
& $\#(\Sigma(M)_0^n\cap E^n)$
& $\ker\xi_{E,N}^{\vee}$
& $\#\ker\xi_{E,N}^{\vee}$
& Index \\
\midrule
$\lmfdbec{32}{a}{4}$ & $256$ & $8$
& $(\Z/2\Z)^2\oplus\Z/4\Z$ & $16$ & $2$ \\
$\lmfdbec{32}{a}{4}$ & $384$ & $32$
& $(\Z/2\Z)^4\oplus\Z/4\Z$ & $64$ & $2$ \\
$\lmfdbec{32}{a}{4}$ & $512$ & $16$
& $(\Z/2\Z)^2\oplus(\Z/4\Z)^2$ & $64$ & $4$ \\
$\lmfdbec{96}{b}{3}$ & $384$ & $1$
& $\Z/2\Z$ & $2$ & $2$ \\
\bottomrule
\end{tabular}
\end{center}

These pairs are the only exceptions for all $N\leq623$. Therefore, we expect that this method is applicable for a vast majority of curves.
\end{remark}

\begin{remark}\label{rem:conjecture-counterexample}
    In \cite[Conjecture 3.13]{DerickxOrlic2024} the authors conjectured that $\Sigma(M)_0^n\cap E^n=\ker \xi_{E,N}^\vee=\{0\}$ for all strong Weil curves $E/\Q$ of odd analytic rank. However, the following counterexamples in Table \ref{tab:noninjective_kernel} show that this conjecture does not hold without further assumptions. Notice that the kernel is a $2$-group in all these examples, as suggested by Theorem \ref{thmling}.

    \begin{longtable}{cccc}
\caption{Non-injective examples and kernel invariants. All these curves have rank $1$ and trivial Shimura subgroup $\Sigma_E$.}
\label{tab:noninjective_kernel}\\

\toprule
LMFDB label of $E$ & $N$ & $N/M$ & kernel \\
\midrule
\endfirsthead

\multicolumn{4}{c}{\tablename\ \thetable\ -- continued from previous page} \\
\toprule
LMFDB label of $E$ & $N$ & $N/M$ & kernel \\
\midrule
\endhead

\midrule
\multicolumn{4}{r}{Continued on next page}
\endfoot

\bottomrule
\endlastfoot

\lmfdbec{288}{b}{3} & 1152 & 4 & $[2]$ \\
\lmfdbec{480}{a}{3} & 1920 & 4 & $[2]$ \\
\lmfdbec{480}{f}{3} & 1920 & 4 & $[2]$ \\
\lmfdbec{672}{e}{2} & 2688 & 4 & $[2]$ \\
\lmfdbec{672}{f}{3} & 2688 & 4 & $[2]$ \\
\lmfdbec{800}{d}{3} & 3200 & 4 & $[2]$ \\
\lmfdbec{288}{b}{3} & 2304 & 8 & $[2,2]$ \\
\lmfdbec{480}{a}{3} & 3840 & 8 & $[2,2]$ \\
\lmfdbec{480}{f}{3} & 3840 & 8 & $[2,2]$ \\
\lmfdbec{672}{e}{2} & 5376 & 8 & $[2,2]$ \\
\lmfdbec{672}{f}{3} & 5376 & 8 & $[2,2]$ \\
\lmfdbec{288}{b}{3} & 3456 & 12 & $[2,2]$ \\
\lmfdbec{800}{d}{3} & 9600 & 12 & $[2,2]$ \\
\lmfdbec{288}{b}{3} & 4608 & 16 & $[2,2,2]$ \\
\lmfdbec{480}{a}{3} & 7680 & 16 & $[2,2,2]$ \\
\lmfdbec{480}{f}{3} & 7680 & 16 & $[2,2,2]$ \\
\lmfdbec{480}{a}{3} & 9600 & 20 & $[2,2]$ \\
\lmfdbec{480}{f}{3} & 9600 & 20 & $[2,2]$ \\
\lmfdbec{288}{b}{3} & 9216 & 32 & $[2,2,2,2]$ \\
\lmfdbec{288}{b}{3} & 18432 & 64 & $[2,2,2,2,2]$ \\

\end{longtable}
\end{remark}

\subsection{The algorithm}\label{subsec:procedure-summary}

We now combine the results of the previous subsections into an algorithm for determining the $D$-elliptic modular curves $X_0(N)$.

The modular degree and oldform criteria give inexpensive necessary conditions, while Proposition~\ref{prop:candidate-criteria} gives the final necessary and sufficient condition for a fixed target. The Shimura subgroup method of Section~\ref{subsec:new_quad_forms} gives a more explicit description of the relevant lattices when its hypotheses are satisfied; otherwise, we use the integral homology lifting criterion.

Fix a positive integer $D$. The algorithm proceeds as follows.

\begin{enumerate}[label=\textup{(\arabic*)},leftmargin=2em]

\item First bound the possible levels $N$. For every prime $p\nmid N$, Ogg's inequality \cite[Lemma~1.3]{Jeon2021} gives
\[
\frac{p-1}{12}\psi(N)+2^{\omega(N)}\leq D(p+1)^2.
\]
For fixed $D$, the usual argument with the smallest prime not dividing $N$, as in \cite[Lemma~3.2]{HasegawaShimura_trig}, gives an effective upper bound on $N$. Hence only finitely many levels need to be considered.

\item For each remaining level $N$, let $M\mid N$ and let $E/\Q$ be a strong Weil curve of conductor $M$. Put $q=\mathcal D_{c_E}(N/M)$. By Theorem~\ref{thm:divisibility}, we may discard $(M,E)$ unless $\deg(\pi_E)\mid qD$.

\item For each remaining pair $(M,E)$ and each $E'\sim_\Q E$, choose a generator $u:E\to E'$ of $\HomQ(E,E')$ and put $\delta=\deg(u)$. By Corollary~\ref{cor:old-form-obstruction}, we may discard $E'$ unless $Q_{E,N}(\mathbf z)=\delta q^2D$ for some $\mathbf z\in\Z^n$. This is only a necessary condition, since a rational old coordinate vector need not define an actual homomorphism to $E'$.

\item For each target surviving the preceding obstructions, determine the lattice $L_{E',N}$ and the corresponding quadratic form $Q_{E',N}$ of Proposition~\ref{prop:candidate-criteria}. If $c_Ec_u=1$, Corollary~\ref{cor:fixed-target-old-basis} applies directly. If $\Sigma_E$ is cyclic and
\[
\Sigma(M)_0^n\cap E^n=\ker\xi_{E,N}^{\vee},
\]
then Proposition~\ref{prop:quad_form_factorisation} and Corollary~\ref{cor:new_basis} give an explicit description in terms of the Shimura subgroup. In the remaining cases, we use the integral-homology criterion of Lemmas~\ref{lem:isogenous-lift} and~\ref{lem:matrix-lift}, together with Lemma~\ref{lem:hermite-lattice-basis}.

\item Finally, determine whether $Q_{E',N}$ represents $D$ over $\Z$. Since $Q_{E',N}$ is positive definite, this is a finite computation, for instance by the Fincke--Pohst algorithm \cite{FinckePohst1985}. By Proposition~\ref{prop:candidate-criteria}, such a representation exists if and only if there is a degree $D$ morphism $X_0(N)\to E'$.

\end{enumerate}

Steps~\textup{(2)} and~\textup{(3)} are preliminary obstructions whose purpose is to reduce the number of targets for which the full lattice has to be computed. The decisive step is the representation problem for $Q_{E',N}$. In particular, the Shimura-subgroup construction is not an additional hypothesis in the general algorithm, but provides a more explicit description of the relevant lattice when its hypotheses are satisfied.

We can now prove Theorem~\ref{thm:main1}.

\begin{proof}[Proof of Theorem~\ref{thm:main1}]
Fix a positive integer $D$. By step~\textup{(1)}, there are only finitely many possible levels $N$, and an upper bound for them can be computed effectively from $D$. For each such $N$, only finitely many elliptic isogeny classes occur in $J_0(N)$, and each such rational isogeny class contains only finitely many elliptic curves up to $\Q$-isomorphism. Steps~\textup{(2)}--\textup{(4)} therefore reduce the problem to finitely many positive-definite quadratic forms $Q_{E',N}$.

By step~\textup{(5)} and Theorem~\ref{thm:pair-to-level}, $X_0(N)$ is $D$-elliptic over $\Q$ if and only if one of these forms represents $D$. Hence the procedure terminates and determines all $D$-elliptic curves $X_0(N)$.
\end{proof}

\section{Classification of $D$-elliptic curves $X_0(N)$ for $D=3,4,5$}\label{sec:positive}

We now apply the methods developed in Section~\ref{sec:degree-pairing} to prove the second part of Theorem~\ref{thm:main2}.

First, we establish the positive $D$-elliptic cases using modular parametrizations, degeneracy maps, and quotient degree forms. Example~\ref{ex:84} illustrates how passing to an isogenous quotient can yield a morphism of smaller degree, while Example~\ref{exam:32a4-256-degrees} illustrates the integral-homology method in a case where the Shimura kernel condition fails. We then combine these constructions with the exclusion criteria to complete the classification for $D=3,4,5$.

\begin{proposition}\label{prop:morphisms}
For every level $N$ appearing in Table~\ref{tab:positive-representative-maps}, there exists a rational morphism from $X_0(N)$ to the indicated elliptic curve $E$ of exact degree $D$.
\end{proposition}
\begin{proof}
We first verify that the target curves appearing in Table~\ref{tab:positive-representative-maps} are indeed elliptic curves.
For curves $X_0(M)$ and quotients by Atkin-Lehner involutions, this follows from the well-known genus formula,
or can be verified computationally. 
For quotients by the operators $S_2$, $V_2$, and $V_3$, this has been proved in \cite{bars22biellipticquotients}.

We now verify that the morphisms listed in Table~\ref{tab:positive-representative-maps} are defined over $\mathbb{Q}$ and have the
indicated degree.

If $E=X_0(M)$, then the corresponding morphism is a degeneracy map. The degree of a degeneracy map
$X_0(N)\to X_0(M)$ is $\psi(N)/\psi(M)$ \cite[Section 2.3]{DerickxOrlic2024}, and it is defined over $\mathbf{Q}$.
 
If $E$ is given by its LMFDB label and $\mathrm{Cond}(E)=N$, then the corresponding morphism is
its modular parametrization. The modular degrees of these curves are recorded in LMFDB.

If $E$ is a quotient by Atkin-Lehner involutions, then the desired rational map is the corresponding
quotient map, which is defined over $\mathbb{Q}$.

If $E$ is a quotient by involutions involving the operators $S_2,V_2,V_3$, then one can check in \cite{bars22biellipticquotients} that the corresponding quotient $X_0(N)/w_d$ is bielliptic. Moreover, the operators $S_2,V_2$ are defined over $\Q$ (\cite[Section 3]{Hasegawa1997}), and the operator $V_3$ defines an involution of $X_0(N)/W$ defined over $\Q$ when $w_9\in W$ (\cite[Proposition 1]{Hasegawa1997}, \cite[Proposition 4.19]{bars22biellipticquotients}).

In addition, one can check in \cite{JEON2018319} that the curve $X_0^+(135)$ is bielliptic and that the corresponding automorphism is defined over $\Q$.

The existence of rational morphisms to curves $E/\Sigma_E$ will be considered in the next proposition by computing quadratic forms.
\end{proof}

\begin{proposition}\label{prop:quad_forms}
    There exists a degree $D$ rational morphism from $X_0(N)$ to an elliptic curve $E/\Sigma_E$ for the following levels $N$:
\begin{center}
\begin{tabular}{c|c|c|c|c|c|c}
\toprule
$N$ & $E$ & $\#\Sigma_E$ & quadratic form for $E$ & quadratic form for $E/\Sigma_E$ & $D$ & vector\\
\hline
$160$ & $\lmfdbec{32}{a}{4}$ & $2$ & $6x^2-4xy+6y^2$ & $12x^2+16xy+8y^2$ & $4$ & $(1,-1)$\\
$38$ & $\lmfdbec{19}{a}{2}$ & $3$ & $3x^2+3y^2$ & $9x^2+6xy+2y^2$ & $5$ & $(1,-1)$\\
$54$ & $\lmfdbec{27}{a}{3}$ & $3$ & $3x^2+3y^2$ & $9x^2+6xy+2y^2$ & $5$ & $(1,-1)$\\
$81$ & $\lmfdbec{27}{a}{3}$ & $3$ & $3x^2+3y^2$ & $9x^2+6xy+2y^2$ & $5$ & $(1,-1)$\\
\bottomrule
\end{tabular}
\end{center}
\end{proposition}

\begin{proof}
    In these cases the Shimura subgroup $E$ is obviously cyclic and we used Sage to prove $\# \Sigma_E=\#\ker\xi_{E,N}^\vee$. Thus we can use Proposition \ref{prop:quad_form_factorisation} and Corollary \ref{cor:new_basis}.
    
    The coefficients of \emph{old quadratic forms} were computed with Theorem \ref{thm:do-degree-pairing}. After that, the quadratic form for $E/\Sigma_E$ can be obtained by multiplying the \emph{old quadratic form} by $\#\Sigma_E$ (this represents composing the maps $\Phi_d$ with the isogeny $u:E\to E/\Sigma_E$) and applying the base change to the new basis given in Corollary \ref{cor:new_basis}.

    For example, consider the case $N=160, E=\lmfdbec{32}{a}{4}$. There the corresponding degree matrix for maps to $E$ is $\begin{pmatrix}
        6 & -2\\ -2 & 6\\
    \end{pmatrix}$. The basis for maps to $E/\Sigma_E$ is $\left\{u\circ\Phi_1, \frac{u\circ\Phi_1-u\circ\Phi_5}{2}\right\}$. Hence, to find all possible degrees of rational maps to $E/\Sigma_E$, we need to multiply that degree matrix by $\deg u=\#\Sigma_E=2$ and conjugate it with $\begin{pmatrix}
        1 & \frac{1}{2}\\0 & \frac{-1}{2} 
    \end{pmatrix}$. We get \[\begin{pmatrix}
        1 & \frac{1}{2}\\0 & \frac{-1}{2} 
    \end{pmatrix}^T\cdot2\cdot\begin{pmatrix}
        6 & -2\\ -2 & 6\\
    \end{pmatrix}\cdot\begin{pmatrix}
        1 & \frac{1}{2}\\0 & \frac{-1}{2} 
    \end{pmatrix}=\begin{pmatrix}
        12 & 8\\ 8 & 8\\
    \end{pmatrix}.\] The quadratic form $12x^2+16xy+8y^2$ attains the value $4$ for the vector $(x,y)=(1,-1)$. Hence, the exact degree $4$ rational map $X_0(160)\to E/\Sigma_E$ is 
    \[
\left(u\circ\Phi_1-
\frac{u\circ\Phi_1-u\circ\Phi_5}{2}\right)\circ j_{160}
=
\left(\frac{u\circ\Phi_1+u\circ\Phi_5}{2}\right)\circ j_{160}.
\]
\end{proof}

The following example, slightly more involved than those in Proposition \ref{prop:quad_forms}, further showcases the usage of quadratic forms to compute all possible degrees of maps to elliptic curves.

\begin{example}\label{ex:84}
Consider the pair $(N,M)=(84,21)$. There is exactly $1$ $\Q$-isogeny class of elliptic curves with conductor $21$ and the strong Weil curve in that class is \(E=\lmfdbec{21}{a}{5}\). We want to determine all possible degrees of maps from $X_0(84)$ to $E$.

We compute that $n=3$, $\#\Sigma_E=2$, and $\#\ker\xi_{E,N}^\vee=2^2=4$. Therefore, the assumptions of Corollary \ref{cor:new_basis} are satisfied. The basis for $\HomQ(J_0(84),E)$ is $\{\Phi_1,\Phi_2,\Phi_4\}$ and the corresponding quadratic form (i.e. \emph{old quadratic form}) is \[6x^2-4xy+6y^2-4xz-4yz+6z^2.\] The basis for $\HomQ(J_0(84),E/\Sigma_E)$ is $\{u\circ\Phi_1,\frac{u\circ\Phi_1-u\circ\Phi_2}{2},\frac{u\circ\Phi_1-u\circ\Phi_4}{2}\}$. Hence the degree matrix for $E/\Sigma_E$ is \[\begin{pmatrix}
        1 & \frac{1}{2} & \frac{1}{2}\\0 & \frac{-1}{2} & 0\\0 & 0 & \frac{-1}{2}
    \end{pmatrix}^T\cdot2\cdot\begin{pmatrix}
        6 & -2 & -2\\ -2 & 6 & -2\\ -2 & -2 & 6
    \end{pmatrix}\cdot\begin{pmatrix}
        1 & \frac{1}{2} & \frac{1}{2}\\0 & \frac{-1}{2} & 0\\0 & 0 & \frac{-1}{2}
    \end{pmatrix}=\begin{pmatrix}
        12 & 8 & 8\\ 8 & 8 & 4\\ 8 & 4 & 8
    \end{pmatrix}\] and the corresponding quadratic form is \[12x^2+16xy+8y^2+16xz+8yz+8z^2.\] It is trivial to check that these two quadratic forms do not attain any odd values because the coefficients are all even. Let us now consider the values $2$ and $4$.

    The first quadratic form can be rewritten as \[2x^2+2y^2+2z^2+2(x-y)^2+2(x-z)^2+2(y-z)^2\] and from there is not hard to see that it cannot attain neither $2$ nor $4$. The second quadratic form can be rewritten as \[4(x+y)^2+4(x+z)^2+4(x+y+z)^2.\] It obviously can never be equal to $2$ and is equal to $4$ when exactly one of the summands is nonzero and equal to $4$. This only happens for vectors $(x,y,z)=(1,-1,-1),(1,-1,0),(1,0,-1)$ (we may without loss of generality assume that $x\geq0$).
    
    Since any map $X_0(N)\to E'$ must factor through either $E$ or $E/\Sigma_E$, we conclude that there are no maps $X_0(N)\to E'$ of degree $3$ or $5$ and that the only degree $4$ maps must be to $E/\Sigma_E$. More precisely, these ones (up to multiplication by $\pm1$): \begin{align*}
        u\circ\Phi_1-\frac{u\circ\Phi_1-u\circ\Phi_2}{2}-\frac{u\circ\Phi_1-u\circ\Phi_4}{2}&=\frac{u\circ\Phi_2+u\circ\Phi_4}{2},\\
        u\circ\Phi_1-\frac{u\circ\Phi_1-u\circ\Phi_2}{2}&=\frac{u\circ\Phi_1+u\circ\Phi_2}{2},\\
        u\circ\Phi_1-\frac{u\circ\Phi_1-u\circ\Phi_4}{2}&=\frac{u\circ\Phi_1+u\circ\Phi_4}{2}.
    \end{align*}

\end{example}

For the exceptional pair \((\lmfdbec{32}{a}{4},256)\) in
Remark \ref{rem:shimura_ker_not_equal}, the hypotheses of
Proposition \ref{prop:quad_form_factorisation} and
Corollary \ref{cor:new_basis} do not hold.
Nevertheless, Ogg's inequality excludes \(D=3,4\), and
the old degree form \(Q_{E,256}\) excludes \(D=5\) by
Corollary \ref{cor:old-form-obstruction}.
For larger degrees, these exclusions need not suffice.
The following example applies the homology criterion of
Lemma \ref{lem:matrix-lift} and the lattice-basis construction of
Lemma \ref{lem:hermite-lattice-basis} to determine
\(L_{E',256}\) and all possible degrees.

\begin{example}
\label{exam:32a4-256-degrees}  
    Let \(E=32.a4\), and put \(N=256\). We use the old-coordinate lattices to determine the possible degrees of nonconstant rational morphisms from \(X_0(N)\) to elliptic curves in the \(\mathbb Q\)-isogeny class of \(E\). We will show that these degrees are precisely the positive multiples of \(8\).

The curve \(E\) is the strong Weil curve in its rational isogeny class and
\(c_E=1\). The rational isogeny class \(\lmfdbeciso{32}{a}\) consists of
the four \(\Q\)-isomorphism classes
\[
E=\lmfdbec{32}{a}{4},\qquad
\lmfdbec{32}{a}{3},\qquad
\lmfdbec{32}{a}{1},\qquad
\lmfdbec{32}{a}{2}.
\]
The old coordinates are indexed by \(1,2,4,8\), and the old
degree-pairing matrix is \(8I_4\). For \(E'=E\), take \(u\) to be the
identity. Since \(c_Ec_u=1\), Corollary
\ref{cor:fixed-target-old-basis} gives
\[
L_{E,256}=\Z^4,
\qquad
Q_{E,256}(z_1,z_2,z_4,z_8)
=8(z_1^2+z_2^2+z_4^2+z_8^2).
\]
Assume first that \(8\mid D\), and write \(D=8m\). By Lagrange's
four-square theorem, there is a vector in \(L_{E,256}=\Z^4\) at which
\(Q_{E,256}\) takes the value \(D\). Proposition
\ref{prop:candidate-criteria} therefore gives a degree \(D\) morphism
from \(X_0(256)\) to \(E\).

Conversely, suppose that a degree \(D\) morphism from \(X_0(256)\) to
an elliptic curve \(E'\sim_\Q E\) exists. Let \(u:E\to E'\) be the
chosen generator and put \(\delta=\deg(u)\). The degrees corresponding
respectively to
\(E,\lmfdbec{32}{a}{3},\lmfdbec{32}{a}{1},\lmfdbec{32}{a}{2}\) are
\(1,2,4,4\). Corollary \ref{cor:old-form-obstruction} requires
\begin{equation}\label{eq:32a-256-old-form-obstruction}
8(z_1^2+z_2^2+z_4^2+z_8^2)=\delta D
\qquad\text{for some }(z_1,z_2,z_4,z_8)\in\Z^4.
\end{equation}
For \(E'=E\), equation
\eqref{eq:32a-256-old-form-obstruction} gives \(8\mid D\). For
\(E'=\lmfdbec{32}{a}{3}\), it gives \(4\mid D\), whereas for
\(E'\in\{\lmfdbec{32}{a}{1},\lmfdbec{32}{a}{2}\}\), it gives
\(2\mid D\). We determine the old-coordinate lattices and the associated
quadratic forms to prove \(8\mid D\) in the three nonidentity cases.

We first compute the old-coordinate lattice for
\(E'=\lmfdbec{32}{a}{3}\). Since \(g(X_0(256))=21\), the group
\(H_1(J_0(256)(\C),\Z)\) has rank \(42\). Choose integral homology
bases, using the modular-symbol basis on \(H_1(E(\C),\Z)\). For
\(d=1,2,4,8\), let \(B_d\in M_{42,2}(\Z)\) be the row-convention
matrix of \((\Phi_d)_*\). Computing \(\hat u_*\) in these bases and
taking the row Hermite normal form of its image gives
\[
U=\begin{pmatrix}1&1\\0&2\end{pmatrix}
\]
for \(E'=\lmfdbec{32}{a}{3}\).
Here \(\delta=2\). For
\(\mathbf x=(x_1,x_2,x_4,x_8)^T\), Lemmas
\ref{lem:isogenous-lift} and \ref{lem:matrix-lift} give
\begin{equation}\label{eq:32a-256-lift-lattice}
\mathbf x\in L_{\lmfdbec{32}{a}{3},256}
\quad\Longleftrightarrow\quad
2(x_1B_1+x_2B_2+x_4B_4+x_8B_8)U^{-1}
\in M_{42,2}(\Z).
\end{equation}

To apply Lemma \ref{lem:hermite-lattice-basis}, we first express
\(2L_{\lmfdbec{32}{a}{3},256}\) as an integral column lattice.  Define
\(C\in M_{84,4}(\Z)\) by
\[
C\mathbf b=
\left(
\left[2\left(\sum_{d\mid8}b_dB_d\right)U^{-1}\right]_{ij}
\right)_{\substack{1\leq i\leq42\\1\leq j\leq2}},
\quad
\mathbf b=(b_1,b_2,b_4,b_8)^{T}\in\Z^4.
\]
The matrix \(C\) is integral because \(2U^{-1}\) is integral.
Equation \eqref{eq:32a-256-lift-lattice} is equivalent to
\(C\mathbf x\in\Z^{84}\).  The nonzero Smith invariant factors of \(C\)
are \(1,2,2,2\), and the kernel of \(C\bmod2\) is
\[
\left\{(\bar b_1,\bar b_2,\bar b_4,\bar b_8)\in\F_2^4\,\middle|\,
\bar b_1+\bar b_2+\bar b_4+\bar b_8=0\right\},
\qquad \F_2=\Z/2\Z.
\]
The invariant factors show that
\(2L_{\lmfdbec{32}{a}{3},256}\subseteq\Z^4\).  For
\(\mathbf b\in\Z^4\), equation
\eqref{eq:32a-256-lift-lattice} identifies
\(\mathbf b\in2L_{\lmfdbec{32}{a}{3},256}\) with
\(C\mathbf b\equiv0\pmod2\).  Hence
\begin{equation}\label{eq:32a-256-lattice-congruence}
2L_{\lmfdbec{32}{a}{3},256}
=\left\{\mathbf b\in\Z^4\,\middle|\,
b_1+b_2+b_4+b_8\equiv0\pmod2\right\}
=G\Z^4,
\quad
G=\begin{pmatrix}
1&0&0&0\\
0&1&0&0\\
0&0&1&0\\
1&1&1&2
\end{pmatrix}.
\end{equation}
Since \(G\) is in the column Hermite normal form defined by
\eqref{eq:column-hermite-normal-form}, Lemma
\ref{lem:hermite-lattice-basis} with \(q=2\) shows that the columns of
\(G/2\) form a \(\Z\)-basis of
\(L_{\lmfdbec{32}{a}{3},256}\).  Under
\(\mathbf x\mapsto u\circ\Phi_{\mathbf x}\), this basis corresponds to
a \(\Z\)-basis of
\(\HomQ(J_0(256),\lmfdbec{32}{a}{3})\).  Proposition
\ref{prop:candidate-criteria} gives the matrix of
\(Q_{\lmfdbec{32}{a}{3},256}\) in this basis as
\begin{equation}\label{eq:32a-256-degree-two-matrix}
A_2=2\left(\frac G2\right)^{\!T}
(8I_4)\left(\frac G2\right)
=\begin{pmatrix}
8&4&4&8\\
4&8&4&8\\
4&4&8&8\\
8&8&8&16
\end{pmatrix}.
\end{equation}
Every diagonal entry of \(A_2\) and every coefficient
\(2(A_2)_{ij}\), for \(i<j\), is divisible by \(8\).
Consequently, \(Q_{\lmfdbec{32}{a}{3},256}\) takes values in
\(8\Z\). Proposition \ref{prop:candidate-criteria} therefore shows that
every degree of a morphism from \(X_0(256)\) to
\(\lmfdbec{32}{a}{3}\) is divisible by \(8\).

For the remaining two curves, \(\delta=4\). Putting the image lattice of
\(\hat u_*\) in row Hermite normal form gives
\[
U_{\lmfdbec{32}{a}{1}}=
\begin{pmatrix}1&3\\0&4\end{pmatrix},
\qquad
U_{\lmfdbec{32}{a}{2}}=
\begin{pmatrix}1&1\\0&4\end{pmatrix}.
\]
For each choice of \(E'\), replace \(2\) and \(U\) in
\eqref{eq:32a-256-lift-lattice} by \(4\) and \(U_{E'}\), respectively.
For both choices of \(E'\), the corresponding coefficient matrix has
Smith invariant factors \(1,2,2,2\), and its kernel modulo \(2\) is
given by \(b_1+b_2+b_4+b_8\equiv0\pmod2\). Consequently,
\(2L_{E',256}=G\Z^4\), where \(G\) is the matrix in
\eqref{eq:32a-256-lattice-congruence}.  Lemma
\ref{lem:hermite-lattice-basis}, again with \(q=2\), shows that the
columns of \(G/2\) form a \(\Z\)-basis of \(L_{E',256}\).  Under
\(\mathbf x\mapsto u\circ\Phi_{\mathbf x}\), this ordered basis maps to
a \(\Z\)-basis of \(\HomQ(J_0(256),E')\).  Since \(\delta=4\), the
matrix of \(Q_{E',256}\) in this basis is
\begin{equation}\label{eq:32a-256-degree-four-matrix}
A_4=4\left(\frac G2\right)^{\! T}
(8I_4)\left(\frac G2\right)=2A_2
\qquad
\left(E'\in
\{\lmfdbec{32}{a}{1},\lmfdbec{32}{a}{2}\}\right).
\end{equation}
Since \(A_4=2A_2\), each of the two remaining quadratic forms takes
values in \(8\Z\). 
Thus a rational morphism of exact degree $D$ from $X_0(256)$
to some elliptic curve $E'\sim_{\Q}E$ exists if and only if $8\mid D$.
\end{example}




We are now ready to present the proofs of theorems on $D$-elliptic curves $X_0(N)$ for $D\in\{3,4,5\}$. The same approach also works for all $3\leq D\leq 100$ and can be automated. The results for all these values $D$ are on \begin{center}
    \url{https://github.com/nt-lib/D-elliptic/tree/main/results}.
\end{center}.
\begin{proof}[Proof of Theorem \ref{thm:main2} for $D\in\{3,4,5\}$]
We will first give an upper bound on possible $D$-elliptic levels $N$ by applying Ogg's inequality \cite[Lemma~1.3]{Jeon2021}. Define
\[
        \psi(N)=N\prod_{q\mid N}\left(1+\frac1q\right),
        \qquad
        \omega(N)=\#\{q\mid q \text{ is prime and } q\mid N\}.
\] Then for a prime $p\nmid N$ we must have \[\frac{p-1}{12}\psi(N)+2^{\omega(N)}\leq|X_0(N)(\F_{p^2})|\leq D\cdot|E_{\F_{p^2}}|\leq D(p+1)^2.\] However, one can check that this inequality does not hold for $N\geq516$, the proof is exactly the same as the proof of \cite[Lemma 3.2]{HasegawaShimura_trig}. Therefore, we only need to consider the levels $N\leq515$.

The levels $N$ listed in the Theorems admit a degree $D$ rational map to an elliptic curve. This was proved in Propositions \ref{prop:morphisms} and \ref{prop:quad_forms} and all these maps are recorded in Table \ref{tab:positive-representative-maps}.

We now use Proposition \ref{prop:quad_form_factorisation} and Corollary \ref{cor:new_basis} to determine the existence of degree $D$ rational maps to an elliptic curve $E/\Q$. In order to be able to use them, we need to check the conditions on Shimura subgroup $\Sigma_E$ and $\ker\xi_{E,N}^\vee$. We do that in Remark \ref{rem:shimura_ker_not_equal} and conclude that there are only $4$ pairs $(E,N)$ when these conditions are not satisfied.

For \(D=3\) and \(D=4\), all three levels \(N=256,384,512\) are eliminated by Ogg's inequality. For \(D=5\), the levels \(N=384,512\) are still ruled out by Ogg's inequality, while for the remaining pair \((E,N)=(\lmfdbec{32}{a}{4},256)\) Corollary \ref{cor:fixed-target-old-basis} immediately determines a basis of \(\Hom_{\mathbb Q}(J_0(256),E)\) from which the nonexistence of degree \(5\) maps follows.

For all other pairs $(E,N)$ we can use Corollary \ref{cor:new_basis} to obtain quadratic forms representing all possible degrees of rational maps $X_0(N)\to E/H$, where $H\leq \Sigma_E$. These forms are positive definite and we can determine if they attain some value $d\mid D$. This can be automated, for example with the Fincke-Pohst algorithm for enumerating integer vectors of small norm \cite{FinckePohst1985}. One can look at some quadratic form examples in the proof of Proposition \ref{prop:quad_forms}. The Sage code for these computations is available on GitHub. 

After that, Proposition \ref{prop:quad_form_factorisation} tells us that a map from $X_0(N)$ to an elliptic curve $E'\sim_\Q E$ must factor through one of the curves $E/H$ and it is easy to check on LMFDB whether there exist rational isogenies of degree $D/d$ from $E/H$ to another curve in its $\Q$-isogeny class.
\end{proof}
\begin{remark}
In our computations with \(g(X_0(N))\ge2\), after the preliminary
exclusions, Corollaries \ref{cor:fixed-target-old-basis} and
\ref{cor:new_basis} suffice for \(1\le D\le5\).
The general lattice computation described in
Lemmas \ref{lem:matrix-lift} and \ref{lem:hermite-lattice-basis}
is first used for \(D=6\), for two targets at
\((N,M)=(256,32)\).
Even for \(D=100\), it is used for only \(27\) triples
\((N,E,E')\), arising from \(12\) pairs \((N,E)\).
\end{remark}
\appendix

\section{Computational verification tables}\label{app:tables}
Table \ref{tab:positive-representative-maps} collects the degree $D$ maps used in
the classification.  In each row, \(D\) is the exact degree, \(N\) is the source
level, and the last two columns give one elliptic target together with a
representative morphism \(X_0(N)\to E\). When the target is \(X_0(M)\), the
morphism is a degeneracy map. Target curves written as \(X_0^*(N)\), \(X_0^+(M)\), or
\(X_0(N)/\langle\cdots\rangle\) are quotient constructions by the indicated
involutions, preceded in some cases by a degeneracy map to a lower level.  The
rows involving \(E/\Sigma_E\) are supplied by Proposition \ref{prop:quad_forms}.

For levels $N$ with the text \textit{quadratic form} we could not find an explicit morphism $X_0(N)\to E$ of one of the above types, thus we prove its existence by inspecting the values attained by the quadratic form.

\scriptsize
\begin{longtable}{@{}c c l p{0.5\textwidth}@{}}
\caption{Representative morphisms for positive levels of genus one or greater.}
\label{tab:positive-representative-maps}\\
\toprule
$D$ & $N$ & target curve $E$ & representative morphism\\
\midrule
\endfirsthead
\toprule
$D$ & $N$ & target curve $E$ & representative morphism\\
\midrule
\endhead
\midrule
\multicolumn{4}{r}{continued on the next page}\\
\endfoot
\bottomrule
\endlastfoot
$3$ & 22 & $X_0(11)$ & $\iota_{1,22,11}$\\
 & 30 & $X_0(15)$ & $\iota_{1,30,15}$\\
 & 33 & $\lmfdbec{33}{a}{2}$ & $X_0(33)\to E$\\
 & 34 & $X_0(17)$ & $\iota_{1,34,17}$\\
 & 38 & $X_0(19)$ & $\iota_{1,38,19}$\\
 & 42 & $X_0(21)$ & $\iota_{1,42,21}$\\
 & 45 & $X_0(15)$ & $\iota_{1,45,15}$\\
 & 52 & $\lmfdbec{52}{a}{2}$ & $X_0(52)\to E$\\
 & 54 & $X_0(27)$ & $\iota_{1,54,27}$\\
 & 57 & $\lmfdbec{57}{c}{2}$ & $X_0(57)\to E$\\
 & 63 & $X_0(21)$ & $\iota_{1,63,21}$\\
 & 72 & $X_0(24)$ & $\iota_{1,72,24}$\\
 & 73 & $\lmfdbec{73}{a}{2}$ & $X_0(73)\to E$\\
 & 81 & $X_0(27)$ & $\iota_{1,81,27}$\\
 & 98 & $X_0(49)$ & $\iota_{1,98,49}$\\
 & 108 & $X_0(36)$ & $\iota_{1,108,36}$\\
\hline
$4$ & 28 & $\lmfdbec{14}{a}{3}$ & $X_0(28)\to X_0(14)\to E$\\
 & 30 & $X_0(30)/\left<w_2,w_5\right>$ & $X_0(30)\to X_0(30)/\left<w_2,w_5\right>$\\
 & 33 & $X_0(11)$ & $\iota_{1,33,11}$\\
 & 34 & $\lmfdbec{34}{a}{3}$ & $X_0(34)\to E$\\
 & 39 & $\lmfdbec{39}{a}{1}$ & $X_0(39)\to E$\\
 & 40 & $X_0^*(40)$ & $X_0(40)\to X_0^*(40)$\\
 & 42 & $X_0(14)$ & $\iota_{1,42,14}$\\
 & 44 & $X_0^+(22)$ & $X_0(44)\to X_0(22)\to X_0^+(22)$\\
 & 45 & $\lmfdbec{45}{a}{6}$ & $X_0(45)\to E$\\
 & 48 & $X_0^*(48)$ & $X_0(48)\to X_0^*(48)$\\
 & 51 & $X_0(17)$ & $\iota_{1,51,17}$\\
 & 52 & $X_0^*(52)$ & $X_0(52)\to X_0^*(52)$\\
 & 55 & $\lmfdbec{55}{a}{2}$ & $X_0(55)\to E$\\
 & 56 & $X_0(14)$ & $\iota_{1,56,14}$\\
 & 57 & $X_0(19)$ & $\iota_{1,57,19}$\\
 & 58 & $X_0^*(58)$ & $X_0(58)\to X_0^*(58)$\\
 & 60 & $X_0(20)$ & $\iota_{1,60,20}$\\
 & 62 & $\lmfdbec{62}{a}{3}$ & $X_0(62)\to E$\\
 & 63 & $X_0^*(63)$ & $X_0(63)\to X_0^*(63)$\\
 & 64 & $\lmfdbec{64}{a}{1}$ & $X_0(64)\to E$\\
 & 65 & $\lmfdbec{65}{a}{2}$ & $X_0(65)\to E$\\
 & 66 & $X_0(66)/\left<w_6,w_{11}\right>$ & $X_0(66)\to X_0(66)/\left<w_6,w_{11}\right>$\\
 & 68 & $X_0^*(68)$ & $X_0(68)\to X_0^*(68)$\\
 & 69 & $\lmfdbec{69}{a}{1}$ & $X_0(69)\to E$\\
 & 70 & $X_0(70)/\left<w_5,w_7\right>$ & $X_0(70)\to X_0(70)/\left<w_5,w_7\right>$\\
 & 72 & $X_0^*(72)$ & $X_0(72)\to X_0^*(72)$\\
 & 74 & $X_0^*(74)$ & $X_0(74)\to X_0^*(74)$\\
 & 75 & $X_0^*(75)$ & $X_0(75)\to X_0^*(75)$\\
 & 76 & $X_0^*(76)$ & $X_0(76)\to X_0^*(76)$\\
 & 77 & $X_0^*(77)$ & $X_0(77)\to X_0^*(77)$\\
 & 78 & $X_0(78)/\left<w_3,w_{13}\right>$ & $X_0(78)\to X_0(78)/\left<w_3,w_{13}\right>$\\
 & 80 & $X_0(20)$ & $\iota_{1,80,20}$\\
 & 82 & $X_0^*(82)$ & $X_0(82)\to X_0^*(82)$\\
 & 84 & $X_0(84)/\left<w_4,S_2w_{14}S_2\right>$ & $X_0(84)\to X_0(84)/w_4\to X_0(84)/\left<w_4,S_2w_{14}S_2\right>$\\
 & 85 & $\lmfdbec{85}{a}{1}$ & $X_0(85)\to E$\\
 & 86 & $X_0^*(86)$ & $X_0(86)\to X_0^*(86)$\\
 & 88 & $X_0^+(44)$ & $X_0(88)\to X_0(44)\to X_0^+(44)$\\
 & 90 & $X_0(90)/\left<w_9,V_3w_{10}\right>$ & $X_0(90)\to X_0(90)/w_9\to X_0(90)/\left<w_9,V_3w_{10}\right>$\\
 & 91 & $X_0^*(91)$ & $X_0(91)\to X_0^*(91)$\\
 & 94 & $\lmfdbec{94}{a}{1}$ & $X_0(94)\to E$\\
 & 96 & $X_0(24)$ & $\iota_{1,96,24}$\\
 & 98 & $X_0^*(98)$ & $X_0(98)\to X_0^*(98)$\\
 & 99 & $X_0^*(99)$ & $X_0(99)\to X_0^*(99)$\\
 & 100 & $X_0^*(100)$ & $X_0(100)\to X_0^*(100)$\\
 & 104 & $X_0(104)/\left<w_{104},V_2\right>$ & $X_0(104)\to X_0^+(104)\to X_0(104)/\left<w_{104},V_2\right>$\\
 & 105 & $X_0(105)/\left<w_5,w_{21}\right>$ & $X_0(105)\to X_0(105)/\left<w_5,w_{21}\right>$\\
 & 108 & $X_0^*(108)$ & $X_0(108)\to X_0^*(108)$\\
 & 109 & $\lmfdbec{109}{a}{1}$ & $X_0(109)\to E$\\
 & 110 & $X_0(110)/\left<w_{10},w_{11}\right>$ & $X_0(110)\to X_0(110)/\left<w_{10},w_{11}\right>$\\
 & 111 & $X_0^*(111)$ & $X_0(111)\to X_0^*(111)$\\
 & 112 & $X_0^+(56)$ & $X_0(112)\to X_0(56)\to X_0^+(56)$\\
 & 117 & $X_0(117)/\left<w_9,V_3w_{117}\right>$ & $X_0(117)\to X_0(117)/w_9\to X_0(117)/\left<w_9,V_3w_{117}\right>$\\
 & 118 & $X_0^*(118)$ & $X_0(118)\to X_0^*(118)$\\
 & 119 & $\lmfdbec{17}{a}{2}$ & $X_0(119)\to X_0^+(119)=\lmfdbec{17}{a}{4}\to E$\\
 & 120 & $X_0(120)/\left<w_{15},V_2w_{40}\right>$ & $X_0(120)\to X_0(120)/w_{15}\to X_0(120)/\left<w_{15},V_2w_{40}\right>$\\
 & 121 & $\lmfdbec{121}{b}{2}$ & $X_0(121)\to E$\\
 & 123 & $X_0^*(123)$ & $X_0(123)\to X_0^*(123)$\\
 & 124 & $X_0^*(124)$ & $X_0(124)\to X_0^*(124)$\\
 & 126 & $X_0(126)/\left<w_9,V_3w_7\right>$ & $X_0(126)\to X_0(126)/w_9\to X_0(126)/\left<w_9,V_3w_7\right>$\\
 & 128 & $X_0(32)$ & $\iota_{1,128,32}$\\
 & 135 & $\lmfdbec{15}{a}{7}$ & $X_0(135)\to X_0^+(135)\to E$\\
 & 136 & $X_0(136)/\left<w_8,V_2w_{136}\right>$ & $X_0(136)\to X_0(136)/w_8\to X_0(136)/\left<w_8,V_2w_{136}\right>$\\
 & 141 & $X_0^*(141)$ & $X_0(141)\to X_0^*(141)$\\
 & 142 & $X_0^*(142)$ & $X_0(142)\to X_0^*(142)$\\
 & 143 & $X_0^*(143)$ & $X_0(143)\to X_0^*(143)$\\
 & 144 & $X_0(36)$ & $\iota_{1,144,36}$\\
 & 145 & $X_0^*(145)$ & $X_0(145)\to X_0^*(145)$\\
 & 147 & $X_0(49)$ & $\iota_{1,147,49}$\\
 & 155 & $X_0^*(155)$ & $X_0(155)\to X_0^*(155)$\\
 & 159 & $X_0^*(159)$ & $X_0(159)\to X_0^*(159)$\\
 & 160 & $\lmfdbec{32}{a}{4}/\Sigma_E$ & quadratic form\\
 & 171 & $X_0(171)/\left<w_9,V_3w_{171}\right>$ & $X_0(171)\to X_0(171)/w_9\to X_0(171)/\left<w_9,V_3w_{171}\right>$\\
 & 176 & $X_0(176)/\left<w_{16},V_2w_{176}\right>$ & $X_0(176)\to X_0(176)/w_{16}\to X_0(176)/\left<w_{16},V_2w_{176}\right>$\\
 & 184 & $X_0(92)/ w_{23}$ & $X_0(184)\to X_0(92)\to X_0(92)/ w_{23}$\\
 & 188 & $X_0(94)/ w_{47}$ & $X_0(188)\to X_0(94)\to X_0(94)/ w_{47}$\\
 \hline
$5$ & 38 & $\lmfdbec{19}{a}{2}/\Sigma_E$ & quadratic form\\
 & 46 & $\lmfdbec{46}{a}{2}$ & $X_0(46)\to E$\\
 & 54 & $\lmfdbec{27}{a}{3}/\Sigma_E$ & quadratic form\\
 & 67 & $\lmfdbec{67}{a}{1}$ & $X_0(67)\to E$\\
 & 75 & $X_0(15)$ & $\iota_{1,75,15}$\\
 & 81 & $\lmfdbec{27}{a}{3}/\Sigma_E$ & quadratic form\\
 & 89 & $\lmfdbec{89}{b}{2}$ & $X_0(89)\to E$\\
 & 100 & $X_0(20)$ & $\iota_{1,100,20}$\\
\end{longtable}
\normalsize


\bibliographystyle{siam}
\bibliography{D-elliptic-modular-curves-X0N-2026-09-21-references}\vspace{0.75in}

@article {Bars1999,
    AUTHOR = {Bars, Francesc},
     TITLE = {Bielliptic modular curves},
   JOURNAL = {J. Number Theory},
  FJOURNAL = {Journal of Number Theory},
    VOLUME = {76},
      YEAR = {1999},
    NUMBER = {1},
     PAGES = {154--165},
      ISSN = {0022-314X,1096-1658},
   MRCLASS = {11G18 (11G30)},
  MRNUMBER = {1688168},
MRREVIEWER = {Joseph\ H.\ Silverman},
       DOI = {10.1006/jnth.1998.2343},
       URL = {https://doi.org/10.1006/jnth.1998.2343},
}

@misc{LMFDB,
  author       = {{The LMFDB Collaboration}},
  title        = {The {$L$}-functions and Modular Forms Database},
  year         = {2026},
  howpublished = {\url{https://www.lmfdb.org}},
  note         = {Accessed September 15, 2026}
}

@manual{sagemath,
  author  = {{The Sage Developers}},
  title   = {{SageMath}, the {Sage} Mathematics Software System},
  version = {10.9},
  year    = {2026},
  note    = {Version 10.9, \url{https://www.sagemath.org}}
}

@article{ogg1974hyperelliptic,
     author = {Ogg, Andrew P.},
     title = {Hyperelliptic modular curves},
     journal = {Bull. Soc. Math. Fr.},
     fjournal = {Bulletin de la Soci\'et\'e Math\'ematique de France},
     pages = {449--462},
     publisher = {Soci\'et\'e math\'ematique de France},
     volume = {102},
     year = {1974},
     doi = {10.24033/bsmf.1789},
     zbl = {0314.10018},
     mrnumber = {51 \#514},
     language = {en},
     url = {www.numdam.org/item/BSMF_1974__102__449_0/}
}

@book{BirkenhakeLange,
  title = {Complex Abelian Varieties},
  ISBN = {9783662063071},
  ISSN = {0072-7830},
  url = {http://dx.doi.org/10.1007/978-3-662-06307-1},
  DOI = {10.1007/978-3-662-06307-1},
  journal = {Grundlehren der mathematischen Wissenschaften},
  publisher = {Springer Berlin Heidelberg},
  author = {Birkenhake,  Christina and Lange,  Herbert},
  year = {2004}
}

@article{Abramovich:Gonality,
  author		= {Dan Abramovich},
  title		= {A linear lower bound on the gonality of modular curves},
  fjournal	= {International Mathematics Research Notices},
  journal     = {Int. Math. Res. Not.},
  year		= {1996},
  pages		= {1005-1011}
}

@article{DHJOdensitydegree,
  author  = {Derickx, Maarten and Hwang, Wontae and
             Jeon, Daeyeol and Orli{\'c}, Petar},
  title   = {Modular curves {$X_0(N)$} of density degree {$5$}},
  journal = {Math. Comp.},
  year    = {2026},
  doi     = {10.1090/mcom/4221},
  note    = {Published online July 27, 2026}
}

@Article{KadetsVogt,
  Author = {Kadets, Borys and Vogt, Isabel},
  Title = {Subspace configurations and low degree points on curves},
  FJournal = {Advances in Mathematics},
  Journal = {Adv. Math.},
  Volume = {460},
  Year = {2025},
  Note = {Article 110021},
  DOI = {10.1016/j.aim.2024.110021},
}

@article{Poonen:Gonality,
  author		= {Bjorn Poonen},
  title		= {Gonality of modular curves in characteristic $p$},
  fjournal	= {Mathematical Research Letters},
  journal     = {Math. Res. Lett},
  volume		= {14},
  year		= {2007},
  number		= {4},
  pages		= {691-701}
}

@book {cohen1,
    AUTHOR = {Cohen, Henri},
     TITLE = {A course in computational algebraic number theory},
    SERIES = {Graduate Texts in Mathematics},
    VOLUME = {138},
 PUBLISHER = {Springer-Verlag},
   ADDRESS = {Berlin},
      YEAR = {1993},
     PAGES = {xii+534},
      ISBN = {3-540-55640-0},
   MRCLASS = {11Y40 (11Rxx 68Q40)},
MRREVIEWER = {Joe P. Buhler},
}

@preamble{
   "\def\cprime{$'$} "
}

@article {derickxVH,
    AUTHOR = {Derickx, Maarten and van Hoeij, Mark},
     TITLE = {Gonality of the modular curve {$X_1(N)$}},
   JOURNAL = {J. Algebra},
  FJOURNAL = {Journal of Algebra},
    VOLUME = {417},
      YEAR = {2014},
     PAGES = {52--71},
      ISSN = {0021-8693},
   MRCLASS = {14G35 (14G25 14H25 14H51)},
MRREVIEWER = {Christian Frederik Wei\ss },
       DOI = {10.1016/j.jalgebra.2014.06.026},
       URL = {https://doi.org/10.1016/j.jalgebra.2014.06.026},
}

@book {Stein2007,
    AUTHOR = {William Arthur Stein},
     TITLE = {Modular forms, a computational approach},
   JOURNAL = {Graduate Studies in Mathematics},
    VOLUME = {79},
     PUBLISHER = {Amer. Math. Soc., Providence, RI},
     YEAR = {2007},
}

@article{AgasheRibetStein2006,
  Author = {Agashe, Amod and Ribet, Kenneth and Stein, William A.},
  Title = {The Manin Constant},
  FJournal = {Pure and Applied Mathematics Quarterly},
  Journal = {Pure Appl. Math. Q.},
  Volume = {2},
  Number = {2},
  Pages = {617--636},
  Year = {2006},
  DOI = {10.4310/PAMQ.2006.V2.N2.A11},
}

@article{Hasegawa1997,
author = {Hasegawa, Yuji},
title = {{Hyperelliptic Modular Curves $X_0^*(N)$}},
volume = {81},
journal = {Acta Arith.},
fjournal = {Acta Arithmetica},
number = {4},
pages = {369 -- 385},
year = {1997},
}

@article{JEON2018319,
title = {{Bielliptic modular curves $X_0^+(N)$}},
journal = {J. Number Theory},
fjournal = {Journal of Number Theory},
volume = {185},
pages = {319-338},
year = {2018},
issn = {0022-314X},
doi = {https://doi.org/10.1016/j.jnt.2017.09.006},
url = {https://www.sciencedirect.com/science/article/pii/S0022314X17303475},
author = {Daeyeol Jeon},
}

@Article{Jeon2021,
 Author = {Jeon, Daeyeol},
 Title = {Modular curves with infinitely many cubic points},
 FJournal = {Journal of Number Theory},
 Journal = {J. Number Theory},
 ISSN = {0022-314X},
 Volume = {219},
 Pages = {344--355},
 Year = {2021},
 Language = {English},
 DOI = {10.1016/j.jnt.2020.09.006},
 zbMATH = {7276963},
 Zbl = {1469.11196}
}

@article{Jeon2022,
  doi = {10.4064/aa220527-7-10},
  url = {https://doi.org/10.4064/aa220527-7-10},
  year = {2022},
  publisher = {Institute of Mathematics,  Polish Academy of Sciences},
  volume = {206},
  number = {2},
  pages = {171--188},
  author = {Daeyeol Jeon},
  title = {{Trielliptic modular curves $X_1(N)$}},
  journal = {Acta Arith.},
  fjournal = {Acta Arithmetica}
}

@article{Jeon2023,
  author  = {Jeon, Daeyeol},
  title   = {Tetraelliptic modular curves {$X_1(N)$}},
  journal = {Acta Arith.},
  volume  = {219},
  number  = {1},
  pages   = {33--51},
  year    = {2025},
  doi     = {10.4064/aa231102-3-2}
}

@Article{JeonPark05,
 Author = {Jeon, Daeyeol and Park, Euisung},
 Title = {Tetragonal modular curves},
 FJournal = {Acta Arithmetica},
 Journal = {Acta Arith.},
 ISSN = {0065-1036},
 Volume = {120},
 Number = {3},
 Pages = {307--312},
 Year = {2005},
 Language = {English},
 DOI = {10.4064/aa120-3-6},
 zbMATH = {5002948},
 Zbl = {1165.11322}
}

@Article{HasegawaShimura_trig,
 Author = {Hasegawa, Yuji and Shimura, Mahoro},
 Title = {Trigonal modular curves},
 FJournal = {Acta Arithmetica},
 Journal = {Acta Arith.},
 ISSN = {0065-1036},
 Volume = {88},
 Number = {2},
 Pages = {129--140},
 Year = {1999},
 Language = {English},
 DOI = {10.4064/aa-88-2-129-140},
 zbMATH = {1310234},
 Zbl = {0947.11018}
}

@article{bars22biellipticquotients,
  Author = {Francesc Bars and Mohamed Kamel and Andreas Schweizer},
  Title = {{Bielliptic quotient modular curves of $X_0(N)$}},
  FJournal = {Mathematics of Computation},
  Journal = {Math. Comp.},
  Volume = {92},
  Number = {340},
  Pages = {895--929},
  Year = {2023},
  DOI = {10.1090/mcom/3800},
}

@article{NajmanOrlic2023,
  Author = {Filip Najman and Petar Orli{\'c}},
  Title = {{Gonality of the modular curve $X_0(N)$}},
  FJournal = {Mathematics of Computation},
  Journal = {Math. Comp.},
  Volume = {93},
  Number = {346},
  Pages = {863--886},
  Year = {2024},
  DOI = {10.1090/mcom/3873},
}

@article{DerickxOrlic2024,
  Author = {Derickx, Maarten and Orli{\'c}, Petar},
  Title = {{Modular curves $X_0(N)$ with infinitely many quartic points}},
  FJournal = {Research in Number Theory},
  Journal = {Res. Number Theory},
  Volume = {10},
  Year = {2024},
  Note = {Article 42},
  DOI = {10.1007/s40993-024-00525-6},
}

@article{DerickxOrlic2026Intermediate,
  title = {Intermediate modular curves with infinitely many quartic points},
  volume = {22},
  ISSN = {1793-7310},
  url = {http://dx.doi.org/10.1142/S1793042126500557},
  DOI = {10.1142/s1793042126500557},
  number = {05},
  journal = {Int. J. Number Theory},
  fjournal = {International Journal of Number Theory},
  publisher = {World Scientific Pub Co Pte Ltd},
  author = {Derickx,  Maarten and Orlić,  Petar},
  year = {2026},
  month = Feb,
  pages = {1053–1074}
}

@article{Orlic2025Intermediate,
  Author = {Orli{\'c}, Petar},
  Title = {Tetragonal intermediate modular curves},
  FJournal = {Acta Arithmetica},
  Journal = {Acta Arith.},
  Volume = {220},
  Number = {3},
  Pages = {263--288},
  Year = {2025},
  DOI = {10.4064/aa240831-22-3},
}

@Inbook{Milne1986AV,
author="Milne, J. S.",
title="Abelian Varieties",
bookTitle="Arithmetic Geometry",
year="1986",
publisher="Springer New York",
address="New York, NY",
pages="103--150",
isbn="978-1-4613-8655-1",
doi="10.1007/978-1-4613-8655-1_7",
url="https://doi.org/10.1007/978-1-4613-8655-1_7"
}

@article{LING199539,
title = {Shimura Subgroups and Degeneracy Maps},
journal = {J. Number Theory},
fjournal = {Journal of Number Theory},
volume = {54},
number = {1},
pages = {39-59},
year = {1995},
issn = {0022-314X},
doi = {https://doi.org/10.1006/jnth.1995.1100},
url = {https://www.sciencedirect.com/science/article/pii/S0022314X85711006},
author = {S. Ling},
}

@article{FinckePohst1985,
 ISSN = {00255718, 10886842},
 URL = {http://www.jstor.org/stable/2007966},
 author = {U. Fincke and M. Pohst},
 journal = {Math. Comp.},
 fjournal = {Mathematics of Computation},
 number = {170},
 pages = {463--471},
 publisher = {American Mathematical Society},
 title = {Improved Methods for Calculating Vectors of Short Length in a Lattice, Including a Complexity Analysis},
 volume = {44},
 year = {1985}
}

@article{Kim2002,
  Author = {Kim, Henry},
  Title = {{Functoriality for the exterior square of $\textup{GL}_4$
    and the symmetric fourth of $\textup{GL}_2$.
    Appendix 2: Refined estimates towards the Ramanujan and Selberg
    conjectures (by Henry Kim and Peter Sarnak)}},
  FJournal = {Journal of the American Mathematical Society},
  Journal = {J. Amer. Math. Soc.},
  Volume = {16},
  Number = {1},
  Pages = {139--183},
  Year = {2003},
  DOI = {10.1090/S0894-0347-02-00410-1},
}

@misc{Cremona:ECData,
  author = {Cremona, John E.},
  title = {Elliptic curve data},
  year = {2019},
  note = {Available at \url{https://johncremona.github.io/ecdata/}; see Table Eight on optimality and the Manin constant, and \texttt{manin.txt}}
}

@misc{JeonKwon2026StrongWeil,
  author = {Jeon, Daeyeol and Kwon, Yongjae},
  title = {Strong {Weil} Degree Divisibility at Higher Levels},
  year = {2026},
  eprint = {2608.06054},
  archiveprefix = {arXiv},
  primaryclass = {math.NT},
  note = {Preprint, available at \url{https://arxiv.org/abs/2608.06054}}
}

\end{document}